\documentclass[a4paper,12pt]{article}

\usepackage{amsmath}
\usepackage{amsthm}
\usepackage{enumerate}
\usepackage{amssymb,amsfonts,latexsym,mathtools}
\usepackage{mathrsfs}
\usepackage{bm}
\usepackage{cases}
\usepackage{color}

\newcommand{\Levy}{L\'{e}vy}

\newcommand{\R}{\mathbb{R}}
\newcommand{\F}{\mathscr{F}}

\newcommand{\Z}{\mathbb{Z}}

\newcommand{\e}{\varepsilon}

\newcommand{\cadlag}{c\`{a}dl\`{a}g}
\renewcommand{\P}{\mathbb{P}}
\usepackage{multirow}

\numberwithin{equation}{section}

\makeatletter
\renewcommand\section{\@startsection {section}{1}{\z@}%
{-3.5ex \@plus -1ex \@minus -.2ex}%
{2.3ex \@plus.2ex}%
{\normalfont\large\bf}}
\makeatother

\makeatletter
\renewcommand\subsection{\@startsection {subsection}{1}{\z@}%
{-3.5ex \@plus -1ex \@minus -.2ex}%
{2.3ex \@plus.2ex}%
{\normalfont\normalsize\bf}}
\makeatother

\theoremstyle{plain}
\newtheorem{thm}{Theorem}[section]
\newtheorem{lem}[thm]{Lemma}

\newtheorem{prop}[thm]{Proposition}
\theoremstyle{definition}
\newtheorem{Rem}[thm]{Remark}
\allowdisplaybreaks

\usepackage[dvipdfmx]{hyperref}
\hypersetup{
setpagesize=false,
 bookmarksnumbered=true,
 bookmarksopen=true,
 colorlinks=true,
 linkcolor=blue,
 citecolor=red,
}

\begin{document}
\begin{center}
\Large \textbf{Conditioning to avoid bounded sets for a one-dimensional \Levy\ processes}
\end{center}
\begin{center}
Kohki IBA\footnote{Graduate School of Science, Osaka University, Japan. E-mail: kohki.iba@gmail.com}
\end{center}
\begin{abstract}
For several classes of bounded sets $A$, the limit of a one-dimensional \Levy\ process conditioned to avoid $A$ up to a parametrized random time which tends to infinity. For $A$ we take the set of finite points with several clocks and a bounded $F_\sigma$-set with exponential clock. We also take an integer lattice with exponential clock.
\end{abstract}

%%%%%%%%%%%%%%%%%%%%%%%%%%%%%%%%%%%%%%%%%%%%%%%%%%%%%%%%%%%%%%%%%%%%%%%%%%%%%%%%%%%%%%%%%%%%%%%%%%%%%%%%%%%%%%%%%%%%%%%%%%%%%%%%%%%%%%%%%%%%%%%%%%%%%%%%%%%%%%%%%%%%%%%%%%%%%%%%%%%%%%%%%%%%%%%%%%%%%%%%%%%%%%%%%%%%%%%%%%%%%%%%%%%%%%%%%%%%%%%%%%%%%%%%%%%%%%%%%%%%%%%%%%%%%%%%%%%%%%%%%%%%%%%%%%%%%%%%

\section{Introduction}
A conditioning problem is to study the long-time limit of the form
\begin{align}
\label{a1}
\lim_{\tau\to\infty} \P_x(\Lambda|\ T_A>\tau)\qquad \text{for}\ \Lambda\in \F_t,
\end{align}
where $((X_t)_{t\ge 0},(\F_t)_{t\ge 0},(\P_x)_{x\in \R})$ is a Markov process, $T_A$ is the first hitting time of a set $A$, and $\tau$ is a net of parametrized random times tending to infinity, called a \emph{clock}. In particular, we call this problem \emph{conditioning to avoid the set $A$} or \emph{conditioning to stay the set $A^c$}. Here, for the random clock $\tau$, we adopt one of the following:
\begin{enumerate}
  \item[(C)] Constant clock: $\tau=t$ as $t\to \infty.$
  \item[(Ex)] Exponential clock: $\tau=(\bm{e}_q)$ as $q\to 0+$, where $\bm{e}_q$ has the exponential distribution with parameter $q>0$ and is independent of $(X_t)$.
  \item[(OH)] One-point hitting time clock: $\tau=(T_c)$ as $c\to \pm \infty$, where $T_c$ is the first hitting time at $c$.
  \item[(TH)] Two-point hitting time clock: $\tau=(T_c\wedge T_{-d})$ as $c,d\to \infty$ and $\frac{d-c}{c+d}\to \gamma\in [-1,1]$, which denote $(c,d)\stackrel{(\gamma)}{\to}\infty$.
  \item[(IL)] Inverse local time clock: $\tau=(\eta_u^c)$ as $c\to \pm \infty$, where $(\eta_u^c)_{u\ge 0}$ is the inverse local time.
\end{enumerate}

To solve this problem, we want to find a non-negative harmonic function $h_A$ for the process killed upon reaching $A$, and a function $\rho(\tau)$ such that
\begin{align}
\lim_{\tau\to \infty}\rho(\tau)\P_{X_t}(T_A>\tau)=h_A(X_t)\ \text{a.s. and in}\ L^1(\P_x)
\end{align}
holds (see, Section \ref{S2-1} for the details). Then, we can express the limit (\ref{a1}) via Doob's $h$-transform with respect to $h_A$:
\begin{align}
\lim_{\tau\to\infty} \P_x(\Lambda|\ T_A>\tau)=\P_x\left[1_A\cdot \frac{h_A(X_t)}{h_A(x)}1_{\{T_A>t\}}\right],
\end{align}
for $x\in \{x\in \R;\ h_A(x)> 0\}$.

This problem has been studied for various processes and sets:
\begin{center}
\begin{tabular}{|c|c|c|c|} \hline
 Paper & Process & Conditioning & Clock \\ \hline\hline
 Knight \cite{Kn} & Brownian motion & 
 \begin{tabular}{c}
  avoid $(-\infty,a)$\\
  avoid $(a,\infty)$\\
  stay $[a,-a]$ 
 \end{tabular}
 & \begin{tabular}{c}
  (C)\\
  (IL)
 \end{tabular}\\ \hline
 Lambert \cite{Lam}& s.n. \Levy\ process &stay $[0,a]$&(Ex)\\ \hline
 \begin{tabular}{c}
  Chaumont \cite{Chaumont}\\
  Choumont-Doney \cite{CD}
 \end{tabular} & \Levy\ process & stay $(0,\infty)$ & \begin{tabular}{c}
  (C)\\
  (Ex)
 \end{tabular}\\ \hline
 Kyprianou et al. \cite{KRS}& subordinator& stay $[a,b]$&(Ex) \\ \hline
 Pant\'{i} \cite{Pa} & \Levy \ process & avoid $\{0\}$ & (Ex)\\ \hline
 \begin{tabular}{c}
    D\"{o}ring et al. \cite{DKW}\\
    Lenthe-Weissmann \cite{LW}
 \end{tabular}
  & stable process & avoid $[a,b]$ & (Ex)\\ \hline
 D\"{o}ring et al. \cite{DWW} &\Levy\ process & avoid $[a,b]$& (Ex) \\ \hline
 \begin{tabular}{c}
    Takeda-Yano \cite{TY}\\
    Takeda \cite{Takeda}
 \end{tabular}
  &\Levy\ process & avoid $\{0\}$ &\begin{tabular}{c}
   (Ex)\\
   (OH)\\
   (TH)\\
   (IL)
 \end{tabular} \\\hline \hline
  \multirow{3}{*}{This paper} & \multirow{3}{*}{\Levy\ process} &\begin{tabular}{c}
    avoid $\{a,b\}$\\
    avoid $\{a_1,...,a_n\}$
  \end{tabular}&
  \begin{tabular}{c}
    (Ex)\\
    (OH)\\
    (TH)\\
    (IL)
  \end{tabular}\\ \cline{3-4}
  &&\begin{tabular}{c}
   avoid bounded $F_\sigma$-sets\\
  avoid $L\Z$
  \end{tabular}
  & (Ex) \\ \hline
\end{tabular}
\end{center}

We consider a one-dimensional \Levy\ process $(X_t)$ which is recurrent and for which every point is regular for itself. For the characteristic exponent $\Psi$ of $X$, i.e., $\P_0[e^{i\lambda X_t}]=e^{-t\Psi(\lambda)},$ we always assume the condition
\begin{align}
  \text{\textbf{(A)}} \int_0^\infty \left|\frac{1}{q+\Psi(\lambda)}\right|d\lambda <\infty\qquad \text{for }q>0.
\end{align}
Then, our main theorems are as follows:

First, we consider conditioning to avoid two points (see, Section \ref{S3} for the details).
\begin{thm}
\label{t1.1}
For distinct points $a,b\in \R$ and $-1\le \gamma\le 1$, we define a function
\begin{align}
\varphi_{a,b}^{(\gamma)}(x):=h^{(\gamma)}(x-a)-\P_x(T_b<T_a)h^{(\gamma)}(b-a),
\end{align}
where $h^{(\gamma)}$ is defined by (\ref{b7}). Then, the following assertions hold:
\begin{align}
   \text{(Ex)}&\ \lim_{q\to 0+}\mathbb{P}_x[F_t|\ \bm{e}_q<T_a\wedge T_b]=\mathbb{P}_x\left[F_t \cdot\frac{\varphi_{a,b}^{(0)}(X_t)}{\varphi_{a,b}^{(0)}(x)}1_{\{T_a\wedge T_b>t\}}\right],\\
    \text{(OH)}&\ \lim_{c\to \pm \infty }\mathbb{P}_x[F_t|\ T_c<T_a\wedge T_b]=\mathbb{P}_x\left[F_t \cdot\frac{\varphi_{a,b}^{(\pm 1)}(X_t)}{\varphi_{a,b}^{(\pm 1)}(x)}1_{\{T_a\wedge T_b>t\}}\right],\\
   \text{(TH)}&\ \displaystyle \lim_{(c,d)\stackrel{(\gamma)}{\to}\infty}\mathbb{P}_x\left[F_t|\ T_c\wedge T_{-d}<T_a\wedge T_b\right]= \mathbb{P}_x\left[F_t\cdot \frac{\varphi_{a,b}^{(\gamma)}(X_t)}{\varphi_{a,b}^{(\gamma)}(x)}1_{\{T_a\wedge T_b>t\}}\right],\\
   \text{(IL)}&\ \displaystyle \lim_{c\to \pm \infty}\mathbb{P}_x[F_t|\ \eta_u^c<T_a\wedge T_b]=\mathbb{P}_x\left[F_t\cdot \frac{\varphi_{a,b}^{(\pm 1)}(X_t)}{\varphi_{a,b}^{(\pm 1)}(x)}1_{\{T_a\wedge T_b>t\}}\right],
\end{align}
for all bounded $\F_t$-measurable functionals $F_t$ and for $x\in \R$ such that the denominator is not zero.
\end{thm}

Next, we consider conditioning to avoid $n$ points (see, Section \ref{S4} for the details).
\begin{thm}
\label{t1.2}
For distinct points $a_1,...,a_n\in \R$, we denote $A_n:=\{a_1,...,a_n\}$. For $-1\le \gamma\le 1$, we define a function
\begin{align*}
\varphi_{A_n}^{(\gamma)}(x):&=h^{(\gamma)}(x-a_n)-\sum_{k=1}^{n-1} h^{(\gamma)}(a_k-a_n)\P_x(T_{a_k}= T_{A_n})\\
&=h^{(\gamma)}(x-a_n)-\P_x[h^{(\gamma)}(X_{T_{A_n}}-a_n)].
\stepcounter{equation}\tag{\theequation}
\end{align*}
Then, the following assertions hold:
\begin{align}
  \text{(Ex)}&\ \displaystyle \lim_{q\to 0+}\mathbb{P}_x[F_t|\ \bm{e}_q<T_{A_n}]=\mathbb{P}_x\left[F_t \cdot\frac{\varphi_{A_n}^{(0)}(X_t)}{\varphi_{A_n}^{(0)}(x)}1_{\{T_{A_n}>t\}}\right],\\
  \text{(OH)}&\  \lim_{c\to \pm \infty}\mathbb{P}_x[F_t|\ T_c<T_{A_n}]=\mathbb{P}_x\left[F_t \cdot\frac{\varphi_{A_n}^{(\pm 1)}(X_t)}{\varphi_{A_n}^{(\pm 1)}(x)}1_{\{T_{A_n}>t\}}\right],\\
  \text{(TH)}&\  \lim_{(c,d)\stackrel{(\gamma)}{\to} \infty}\mathbb{P}_x[F_t|\ T_c\wedge T_{-d}<T_{A_n}]=\mathbb{P}_x\left[F_t \cdot\frac{\varphi_{A_n}^{(\gamma)}(X_t)}{\varphi_{A_n}^{(\gamma)}(x)}1_{\{T_{A_n}>t\}}\right],\\
  \text{(IL)}&\  \lim_{c\to \pm \infty}\mathbb{P}_x[F_t|\ \eta_u^c<T_{A_n}]=\mathbb{P}_x\left[F_t \cdot\frac{\varphi_{A_n}^{(\pm 1)}(X_t)}{\varphi_{A_n}^{(\pm 1)}(x)}1_{\{T_{A_n}>t\}}\right],
\end{align}
for all bounded $\F_t$-measurable functionals $F_t$ and for $x\in \R$ such that the denominator is not zero.
\end{thm}

Next, we consider conditioning to avoid bounded $F_\sigma$-sets (see, Section \ref{S5} for the details).
\begin{thm}
\label{t1.3}
Let $A$ is a bounded $F_\sigma$-set which contains $0$. We define a function
\begin{align}
\varphi_A(x):=h(x)-\P_x[h(X_{T_A})].
\end{align}
Then, the following assertion holds:
\begin{align}
\lim_{q\to 0+}\mathbb{P}_x[F_t|\ \bm{e}_q<T_A]=\mathbb{P}_x\left[F_t \cdot\frac{\varphi_A(X_t)}{\varphi_A(x)}1_{\{T_A>t\}}\right]
\end{align}
for all bounded $\F_t$-measurable functionals $F_t$ and for $x\in \R$ such that the denominator is not zero.
\end{thm}

Finally, we consider conditioning to avoid an integer lattice (see, Section \ref{S6} for the details). For the proof, we use a result of Isozaki \cite{Iso}.
\begin{thm}
\label{1.4}
Let $L>0$ and set $L\Z:=\{Ln\in \R;\ n\in \Z\}$. We further assume the following:
\begin{align}
\sum_{n\in \Z}\left|\frac{1}{q+\Psi(\frac{2n\pi}{L})}\right|<\infty\qquad \text{for}\ q>0.
\end{align}
We define a function
\begin{align}
\varphi_{L\Z}(x):=\sum_{n\in \Z\setminus \{0\}}\frac{1-e^{\frac{2n\pi x}{L}i}}{\Psi(\frac{2n\pi}{L})}.
\end{align}
Then, the following assertion holds:
\begin{align}
\displaystyle \lim_{q\to 0+}\mathbb{P}_x[F_t,\ t<\bm{e}_q|\ \bm{e}_q<T_{L\Z}]=\mathbb{P}_x\left[F_t \cdot\frac{\varphi_{L\Z}(X_t)}{\varphi_{L\Z}(x)}1_{\{T_{L\Z}>t\}}\right]
\end{align}
for all bounded $\F_t$-measurable functionals $F_t$ and for $x\in \R$ such that the denominator is not zero.
\end{thm}

\subsection*{Organization}
This paper is organized as follows. In Section \ref{S2}, we prepare some general results of \Levy\ processes. In Section \ref{S2-1}, we prove general theorem for conditioning problems. In Sections \ref{S3}, \ref{S4}, \ref{S5}, and \ref{S6}, we discuss the conditioning to avoid a two points, a $n$ points, bounded $F_\sigma$-sets, and an integer lattice, respectively.

%%%%%%%%%%%%%%%%%%%%%%%%%%%%%%%%%%%%%%%%%%%%%%%%%%%%%%%%%%%%%%%%%%%%%%%%%%%%%%%%%%%%%%%%%%%%%%%%%%%%%%%%%%%%%%%%%%%%%%%%%%%%%%%%%%%%%%%%%%%%%%%%%%%%%%%%%%%%%%%%%%%%%%%%%%%%%%%%%%%%%%%%%%%%%%%%%%%%%%%%%%%%%%%%%%%%%%%%%%%%%%%%%%%%%%%%%%%%%%%%%%%%%%%%%%%%%%%%%%%%%%%%%%%%%%%%%%%%%%%%%%%%%%%%%%%%%%%%%%%%%%%%%%%%%%%%%%%%%%%%%%%%%%%%%%%%

\section{Preliminaries}
\label{S2}
Let $((X_t)_{t\ge 0},(\P_x)_{x\in \R})$ be the canonical representation of a one-dimensional \Levy\ process with $X_0=x$, $\P_x$-a.s. For $t\ge 0$, we denote by $\F_t^X:=\sigma(X_s,\ 0\le s\le t)$ and write $\F_t=\bigcap_{s>t}\F_s^X.$ For a set $A\subset \R$, let $T_A$ be the first hitting time of $A$ for $(X_t)$, i.e.,
\begin{align}
T_A:=\inf\{t\ge 0;\ X_t\in A\}.
\end{align}
For simplicity, we denote $T_{\{a\}}$ as $T_a$ for $a\in \R.$ For $\lambda\in \R$, we denote by $\Psi(\lambda)$ the characteristic exponent of $(X_t)$, i.e., $\Psi(\lambda)$ is defined by
\begin{align}
\P_0\left[e^{i\lambda X_t}\right]=e^{-t\Psi(\lambda)}\qquad \text{for}\ t\ge 0.
\end{align}

Throughout this paper, we always assume $((X_t),\P_0)$ is recurrent, and always assume the condition
\begin{align}
\textbf{(A)}\ \int_0^\infty\left|\frac{1}{q+\Psi(\lambda)}\right|d\lambda<\infty\qquad \text{for}\ q>0.
\end{align}
Then, $(X_t)$ has a bounded continuous resolvent density $r_q$. It is defined by
\begin{align}
\int_{\R}f(x)r_q(x)dx=\P_0\left[\int_0^\infty e^{-qt}f(X_t)dt\right]\qquad \text{for}\ q>0\ \text{and}\ f\ge 0
\end{align}
(see, e.g., Theorems II.16 and II.19 of \cite{Ber}). It is known that the equation between the first hitting time of $a$ and the resolvent density:
\begin{align}
\label{b2}
\P_x\left[e^{-qT_a}\right]=\frac{r_q(a-x)}{r_q(0)}\qquad \text{for}\ q>0\ \text{and}\ x\in \R
\end{align}
(see, e.g., Corollary II.18 of \cite{Ber}).

We define 
\begin{align}
h_q(x):=r_q(0)-r_q(-x)\qquad \text{for}\ q>0\ \text{and}\ x\in \R.
\end{align}
It is clear that $h_q(0)=0$, and by (\ref{b2}), we have $h_q\ge 0.$ The following proposition plays a key role in our results:

\begin{prop}[Theorem 1.1 and Theorem 1.2 of \cite{TY}]
\label{0-2.1}
The following assertions hold:
\begin{enumerate}
  \item For any $x\in \R$, the limit $h(x):=\lim_{q\to 0+}h_q(x)$ exists and is finite;
  \item $h$ is continuous and subadditive on $\R.$
  \item It holds that
  \begin{align}
  \label{0-27}
    \lim_{x\to \pm \infty}\frac{h(x)}{|x|}=\frac{1}{m^2}\in [0,\infty),
  \end{align}
  where $m^2:=\P_0[X_1^2]$ is a variance of $X_1$.
\end{enumerate}
\end{prop}

We call $h$ the \emph{renormalized zero resolvent}. Further, we define
\begin{align}
\label{b7}
h^{(\gamma)}(x):=h(x)+\frac{\gamma}{m^2}x\qquad \text{for}\ -1\le \gamma\le 1.
\end{align}
Note that when $m^2=\infty$, $h^{(\gamma)}$ is independent of $\gamma.$

We can define a local time at $a\in \R$, which we denote by $(L_t^a)_{t\ge 0}$. It is defined by
\begin{align}
L_t^a:=\lim_{\e\to 0+}\frac{1}{2\e}\int_0^t1_{\{|X_s-a|<\e\}}ds.
\end{align}
It is known that $(L_t^a)$ is continuous in $t$ and satisfies 
\begin{align}
\mathbb{P}_x\left[\int_0^\infty e^{-qt}dL_t^a\right]=r_q(a-x)\qquad \text{for $q>0$ and $x\in \R$}
\end{align}
(see, e.g., Section V of \cite{Ber}). In particular, from this expression, $r_q(x)$ is non-decreasing as $q\to 0+.$ Moreover, we know that
\begin{align}
\label{b8}
\lim_{q\to 0+}qr_q(0)=0
\end{align}
(see, e.g., Lemma 15.5 of \cite{Tukada}).

Let $(\eta_l^a)_{l\ge 0}$ be an inverse local time, i.e.,
\begin{align}
\eta_l^a:=\inf\{t>0:\ L_t^a>l\}.
\end{align}
It is known that the process $((\eta_l^a),\mathbb{P}_a)$ is a possibly killed subordinator which has the Laplace exponent
\begin{align}
\mathbb{P}_a\left[e^{-q\eta_l^a}\right]=e^{-\frac{l}{r_q(0)}}\qquad \text{for $l>0$ and $q>0$}
\end{align}
(see, e.g., Proposition V.4 of \cite{Ber}).

%%%%%%%%%%%%%%%%%%%%%%%%%%%%%%%%%%%%%%%%%%%%%%%%%%%%%%%%%%%%%%%%%%%%%%%%%%%%%%%%%%%%%%%%%%%%%%%%%%%%%%%%%%%%%%%%%%%%%%%%%%%%%%%%%%%%%%%%%%%%%%%%%%%%%%%%%%%%%%%%%%%%%%%%%%%%%%%%%%%%%%%%%%%%%%%%%%%%%%%%%%%%%%%%%%%%%%%%%%%%%%%%%%%%%%%%%%%%%%%%%%

\section{General theorems for the conditioning problems}
In this section, we prove a general theory of conditioning problems, since the discussion of the conditioning problem for each set has some common aspects.

First, we consider the case exponential clock.
\label{S2-1}
\begin{thm}
\label{t2}
Let $T$ be a stopping time. We assume that there exists the non-trivial and non-negative function $\varphi_e^T$ which satisfies the following:
\begin{align}
\lim_{q\to 0+}r_q(0)\P_{X_t}(T>\bm{e}_q)=\varphi_e^T(X_t)\ a.s.\ and\ in\ L^1(\P_x).
\end{align}
Then, $(\varphi_e^T(X_t)1_{\{T>t\}})_{t\ge 0}$ is a non-negative $((\F_t),\P_x)$-martingale, and if $\varphi_e^T(x)>0$, then it holds that
\begin{align}
\lim_{q\to 0+}\mathbb{P}_x[F_t|\ \bm{e}_q<T]=\mathbb{P}_x\left[F_t \cdot\frac{\varphi_e^T(X_t)}{\varphi_e^T(x)}1_{\{T>t\}}\right]
\end{align}
for all bounded $\F_t$-measurable functionals $F_t$.
\end{thm}
\begin{proof}
We define 
\begin{align}
N_{t}^{q}:&= r_q(0)\mathbb{P}_x(T>\bm{e}_q>t|\F_t),\\
M_{t}^{q}:&=r_q(0)\mathbb{P}_x\left(\bm{e}_q<T|\F_t\right),
\end{align}
for $q>0.$ By the lack of memory property of the exponential distribution and by the Markov property, we have
\begin{align*}
N_{t}^q&=r_q(0)1_{\{T>t\}}\P_x^{\F_t}(T>\bm{e}_q|\ \bm{e}_q>t)\P_x(\bm{e}_q>t)\\
&=r_q(0)1_{\{T>t\}}\P_{X_t}(T>\bm{e}_q)e^{-qt},
\stepcounter{equation}\tag{\theequation}
\end{align*}
where $\P_x^{\F_t}$ means the conditional probability with respect to $\F_t$. Therefore, we obtain by the hypothesis,
\begin{align}
\lim_{q\to 0+}N_{t}^{q}=\varphi_e^T(X_t)1_{\{T>t\}}\ \mathrm{a.s.\ and\ in}\ L^1(\mathbb{P}_x).
\end{align}
By (\ref{b8}), we have 
\begin{align*}
M_{t}^q-N_{t}^q&=r_q(0)\P_x(T>\bm{e}_q,\ t\ge \bm{e}_q|\F_t)\\
&=qr_q(0)\int_0^t 1_{\{T>s\}}e^{-qs}ds\\
&\le qr_q(0) t\to 0,
\stepcounter{equation}\tag{\theequation}
\end{align*}
as $q\to 0+.$ Therefore, we have
\begin{align}
\lim_{q\to 0+}M_{t}^{q}=\varphi_e^T(X_t)1_{\{T>t\}}\ \mathrm{a.s.\ and\ in}\ L^1(\mathbb{P}_x).
\end{align}
Since $(M_t^q)_{t\ge 0}$ is a non-negative martingale, its $L^1$-limit $(\varphi_e^T(X_t)1_{\{T>t\}})$ is also a non-negative $((\F_t),\P_x)$-martingal (see, e.g., Proposition 1.3 of \cite{CW}).

Finally, we obtain by the $L^1$-convergence and the martingale property,
\begin{align*}
\P_x[F_t|\ \bm{e}_q<T]&=\frac{\P_x[F_t,\ \bm{e}_q<T]}{\P_x(\bm{e}_q<T)}=\frac{\P_x[F_t\cdot M_t^q]}{\P_x[M_t^q]}\\
&\to \frac{\P_x[F_t\cdot  \varphi_e^T(X_t)1_{\{T>t\}}]}{\P_x[\varphi_e^T(X_t)1_{\{T>t\}}]}=\mathbb{P}_x\left[F_t \cdot\frac{\varphi_e^T(X_t)}{\varphi_e^T(x)}1_{\{T>t\}}\right],
\stepcounter{equation}\tag{\theequation}
\end{align*}
as $q\to 0+.$
\end{proof}

The same general theorem is given below for clocks other than exponential clock, but note that there are a few more assumptions.

\begin{thm}
\label{t3}
Let $T$ be a stopping time. We assume that there exists the non-trivial and non-negative function $\varphi_o^T$ which satisfies the following:
\begin{align}
\lim_{c\to \pm \infty}h^B(c)\P_{X_t}(T>T_c)=\varphi_o^T(X_t)\ a.s.\ and\ in\ L^1(\P_x),
\end{align}
where $h^B(c):=\P_0[L_{T_c}^0]=h(c)+h(-c)$ (see, Lemma 3.5 of \cite{TY}). Moreover, we also assume that $\frac{\varphi_e^T(x)}{x}$ converges as $|x|\to \infty.$ Then, $(\varphi_o^T(X_t)1_{\{T>t\}})_{t\ge 0}$ is a non-negative $((\F_t),\P_x)$-martingale, and if $\varphi_o^T(x)>0$, then it holds that
\begin{align}
\lim_{c\to \pm \infty}\mathbb{P}_x[F_t|\ T_c<T]=\mathbb{P}_x\left[F_t \cdot\frac{\varphi_o^T(X_t)}{\varphi_o^T(x)}1_{\{T>t\}}\right]
\end{align}
for all bounded $\F_t$-measurable functionals $F_t$.
\end{thm}
\begin{proof}
We define 
\begin{align}
N_{t}^{c}:&= h^B(c)\mathbb{P}_x(T>T_c>t|\F_t),\\
M_{t}^{c}:&=h^B(c)\mathbb{P}_x\left(T_c<T|\F_t\right),
\end{align}
for $c\in \R.$ By the Markov property, we have
\begin{align*}
N_{t}^c=h^B(c)1_{\{T>t\}}\P_{X_t}(T>T_c).
\stepcounter{equation}\tag{\theequation}
\end{align*}
Therefore, we obtain by the hypothesis,
\begin{align}
\lim_{c\to \pm \infty}N_{t}^{c}=\varphi_o^T(X_t)1_{\{T>t\}}\ \mathrm{a.s.\ and\ in}\ L^1(\mathbb{P}_x).
\end{align}
We have 
\begin{align*}
M_{t}^c-N_{t}^c&=h^B(c)\P_x(T>T_c,\ t\ge T_c|\F_t)\\
&=h^B(c)\P_x(T>T_c|\F_t)1_{\{t\ge T_c\}}\to 0\ a.s.,
\stepcounter{equation}\tag{\theequation}
\end{align*}
as $c\to \pm \infty.$ Since $(\varphi_e^T(X_t)1_{\{T>t\}})$ is a martingale, we have by the optional sampling theorem, by the hypothesis, and by (\ref{0-27}),
\begin{align*}
\P_x\left[|M_t^c-N_t^c|\right]&=h^B(c)\P_x(T>T_c,\ t\ge T_c)\\
&=\frac{h^B(c)}{c}\frac{c}{\varphi_e^T(c)}\P_x[\varphi_e^T(X_{T_c})1_{\{T>T_c\}},\ t\ge T_c]\\
&= \frac{h^B(c)}{c}\frac{c}{\varphi_e^T(c)}\P_x[\varphi_e^T(X_t)1_{\{T>t\}},\ t\ge T_c]\to 0,
\stepcounter{equation}\tag{\theequation}
\end{align*}
as $c\to \pm \infty$. Therefore, we have
\begin{align}
\lim_{c\to \pm \infty}M_{t}^{c}=\varphi_o^T(X_t)1_{\{T>t\}}\ \mathrm{a.s.\ and\ in}\ L^1(\mathbb{P}_x).
\end{align}
The rest of the proof is the same as in Theorem \ref{t2}.
\end{proof}

Next, we consider the case two-point hitting time clock.

\begin{thm}
\label{t4}
Let $T$ be a stopping time. We assume that there exists the non-trivial and non-negative function $\varphi_t^T$ which satisfies the following:
\begin{align}
\lim_{(c,d)\stackrel{(\gamma)}{\to} \infty}h^C(c,-d)\P_{X_t}(T>T_c\wedge T_{-d})=\varphi_t^T(X_t)\ a.s.\ and\ in\ L^1(\P_x),
\end{align}
where $h^C(c,-d):=\P_0[L_{T_c\wedge T_{-d}}^0]$ which can be expressed only by $h$ (see, Lemma 6.1 of \cite{TY}). Moreover, we also assume that $\frac{\varphi_e^T(x)}{x}$ converges as $|x|\to \infty.$ Then, $(\varphi_t^T(X_t)1_{\{T>t\}})_{t\ge 0}$ is a non-negative $((\F_t),\P_x)$-martingale, and if $\varphi_t^T(x)>0$, then it holds that
\begin{align}
\lim_{(c,d)\stackrel{(\gamma)}{\to} \infty}\mathbb{P}_x[F_t|\ T_c\wedge T_{-d}<T]=\mathbb{P}_x\left[F_t \cdot\frac{\varphi_t^T(X_t)}{\varphi_t^T(x)}1_{\{T>t\}}\right]
\end{align}
for all bounded $\F_t$-measurable functionals $F_t$.
\end{thm}
\begin{proof}
The proof is the same as in Theorem \ref{t3}, so we omit it.
\end{proof}

Finally, we consider the case inverse local time clock.

\begin{thm}
\label{t5}
Let $T$ be a stopping time. We assume that there exists the non-trivial and non-negative function $\varphi_i^T$ which satisfies the following:
\begin{align}
\lim_{c\to \pm \infty}h^B(c)\P_{X_t}(T>\eta_u^c)=\varphi_i^T(X_t)\ a.s.\ and\ in\ L^1(\P_x).
\end{align}
Moreover, we also assume that $\frac{\varphi_e^T(x)}{x}$ converges as $|x|\to \infty.$ Then, $(\varphi_i^T(X_t)1_{\{T>t\}})_{t\ge 0}$ is a non-negative $((\F_t),\P_x)$-martingale, and if $\varphi_i^T(x)>0$, then it holds that
\begin{align}
\lim_{c\to \pm \infty}\mathbb{P}_x[F_t|\ \eta_u^c<T]=\mathbb{P}_x\left[F_t \cdot\frac{\varphi_i^T(X_t)}{\varphi_i^T(x)}1_{\{T>t\}}\right]
\end{align}
for all bounded $\F_t$-measurable functionals $F_t$.
\end{thm}
\begin{proof}
Note that for each $u\ge 0,\ c\in \R$, $\eta_u^c$ is a stopping time (see, e.g., Proposition IV.7 of \cite{Ber}). We define 
\begin{align}
N_{t}^{c}:&= h^B(c)\mathbb{P}_x(T>\eta_u^c>t|\F_t),\\
M_{t}^{c}:&=h^B(c)\mathbb{P}_x\left(\eta_u^c<T|\F_t\right),
\end{align}
for $c\in \R.$ By the Markov property, we have
\begin{align*}
N_{t}^c&=h^B(c)1_{\{T>t\}}\P_{X_t}(T>\eta_u^c)\\
&=h^B(c)1_{\{T>t\}}\P_{X_t}(T>T_c)\P_c(T>\eta_u^c).
\stepcounter{equation}\tag{\theequation}
\end{align*}
Therefore, we obtain by the hypothesis,
\begin{align}
\lim_{c\to \pm \infty}N_{t}^{c}=\varphi_i^T(X_t)1_{\{T>t\}}\ \mathrm{a.s.\ and\ in}\ L^1(\mathbb{P}_x).
\end{align}
We have 
\begin{align*}
M_{t}^c-N_{t}^c&=h^B(c)\P_x(T>\eta_u^c,\ t\ge \eta_u^c|\F_t)\\
&=h^B(c)\P_x(T>\eta_u^c|\F_t)1_{\{t\ge \eta_u^c\}}\to 0\ a.s.,
\stepcounter{equation}\tag{\theequation}
\end{align*}
as $q\to 0+.$ Since $(\varphi_e^T(X_t)1_{\{T>t\}})$ is a martingale, we have by the optional sampling theorem, by the hypothesis, and by (\ref{0-27}),
\begin{align*}
\P_x\left[|M_t^c-N_t^c|\right]&=h^B(c)\P_x(T>\eta_u^c,\ t\ge \eta_u^c)\\
&=\frac{h^B(c)}{c}\frac{c}{\varphi_e^T(c)}\P_x[\varphi_e^T(X_{\eta_u^c})1_{\{T>\eta_u^c\}},\ t\ge \eta_u^c]\\
&= \frac{h^B(c)}{c}\frac{c}{\varphi_e^T(c)}\P_x[\varphi_e^T(X_t)1_{\{T>t\}},\ t\ge \eta_u^c]\to 0.
\stepcounter{equation}\tag{\theequation}
\end{align*}
Therefore, we have
\begin{align}
\lim_{c\to \pm \infty}M_{t}^{c}=\varphi_i^T(X_t)1_{\{T>t\}}\ \mathrm{a.s.\ and\ in}\ L^1(\mathbb{P}_x).
\end{align}
The rest of the proof is the same as in Theorem \ref{t3}, so we omit it.
\end{proof}

%%%%%%%%%%%%%%%%%%%%%%%%%%%%%%%%%%%%%%%%%%%%%%%%%%%%%%%%%%%%%%%%%%%%%%%%%%%%%%%%%%%%%%%%%%%%%%%%%%%%%%%%%%%%%%%%%%%%%%%%%%%%%%%%%%%%%%%%%%%%%%%%%%%%%%%%%%%%%%%%%%%%%%%%%%%%%%%%%%%%%%%%%%%%%%%%%%%%%%%%%%%%%%%%%%%%%%%%%%%%%%%%%%%%%%%%%%%%%%%%%%%%%%%%%%%%%%%%%%%%%%%%%%%%%%%%%%%%%%%%%%%%%%%%%%%%%%%%%%%%%%%%%%%%

\section{Conditionings to avoid two points}
\label{S3}
\subsection{Conditioning theorems}
We define for distinct points $a,b\in \R$ and for a constant $-1\le \gamma\le 1,$
\begin{align*}
\label{c1}
\varphi^{(\gamma)}_{a,b}(x):&=h^{(\gamma)}(x-a)-\mathbb{P}_x(T_b<T_a)h^{(\gamma)}(b-a)\\
&=h^{(\gamma)}(x-b)-\mathbb{P}_x(T_a<T_b)h^{(\gamma)}(a-b).
\stepcounter{equation}\tag{\theequation}
\end{align*}

First, we prove of the a.s. convergence.

\begin{thm}
\label{3.1}
The following assertions hold:
\begin{align}
\label{3.1.1}
\text{(Ex)}&\ \lim_{q\to 0+}r_q(0)\P_x(T_a\wedge T_b>\bm{e}_q)=\varphi_{a,b}^{(0)}(x),\\
\text{(OH)}&\ \lim_{c\to \pm\infty}h^B(c)\P_x(T_a\wedge T_b>T_c)=\varphi_{a,b}^{(\pm 1)}(x),\\
\label{3.1.3}
\text{(TH)}&\ \lim_{(c,d)\stackrel{(\gamma)}{\to} \infty}h^C(c,-d)\P_x(T_a\wedge T_b>T_c\wedge T_{-d})=\varphi_{a,b}^{(\gamma)}(x),\\
\label{3.1.4}
\text{(IL)}&\ \lim_{c\to \pm\infty}h^B(c)\P_x(T_a\wedge T_b>\eta_u^c)=\varphi_{a,b}^{(\pm 1)}(x).
\end{align}
\end{thm}
\begin{proof}
The cases of the exponential clock and the one-point hitting time clock are already shown in (3.25) and (4.16) of \cite{IY}, respectively.

Next, we prove the case of the two-point hitting time clock. We have
\begin{align*}
\label{4.6}
&h^C(c,-d)\P_x(T_a\wedge T_b>T_c\wedge T_{-d})\\
&\qquad =h^C(c,-d)\P_x(T_a\wedge T_b\wedge T_{-d}>T_c)+h^C(c,-d)\P_x(T_a\wedge T_b\wedge T_{c}>T_{-d}).
\stepcounter{equation}\tag{\theequation}
\end{align*}
By the strong Markov property, we have
\begin{align*}
&\P_x(T_a\wedge T_b\wedge T_{-d}>T_c)\\
&\qquad =\P_x(T_a\wedge T_b>T_c)-\P_x(T_a\wedge T_b>T_c>T_{-d})\\
&\qquad =\P_x(T_a\wedge T_b>T_c)-\P_{-d}(T_a\wedge T_b>T_c)\P_x(T_a\wedge T_b\wedge T_c>T_{-d})\\
&\qquad =\P_x(T_a\wedge T_b>T_c)-\P_{-d}(T_a\wedge T_b>T_c)\\
&\qquad \qquad \times \{\P_x(T_a\wedge T_b>T_{-d})-\P_c(T_a\wedge T_b>T_{-d})\P_x(T_a\wedge T_b\wedge T_{-d}>T_c)\}.
\stepcounter{equation}\tag{\theequation}
\end{align*}
Thus, we have
\begin{align*}
\label{4.8}
\P_x(T_a\wedge T_b\wedge T_{-d}>T_c)=\frac{\P_x(T_a\wedge T_b>T_c)-\P_{-d}(T_a\wedge T_b>T_c)\P_x(T_a\wedge T_b>T_{-d})}{1-\P_{-d}(T_a\wedge T_b>T_c)\P_c(T_a\wedge T_b>T_{-d})}.
\stepcounter{equation}\tag{\theequation}
\end{align*}
By (4.17) and (5.11) of \cite{IY}, we have
\begin{align*}
\label{4.9}
h^C(c,-d)\P_x(T_a\wedge T_b>T_c)&=\frac{h^C(c,-d)}{h^B(c)}\cdot h^B(c)\P_x(T_a\wedge T_b>T_c)\\
&=\frac{h^C(c,-d)}{h^B(c)}\cdot \frac{h^B(c)\P_x(T_c<T_a)-\P_x(T_b<T_a)\cdot h^B(c)\P_b(T_c<T_a)}{1-\P_c(T_b<T_a)\P_b(T_c<T_a)}\\
&\to \frac{1+\gamma}{2}\left\{h^{(+1)}(x-a)-\P_x(T_b<T_a)h^{(+1)}(b-a)\right\},
\stepcounter{equation}\tag{\theequation}
\end{align*}
as $(c,d)\stackrel{(\gamma)}{\to} \infty.$ Therefore, by (\ref{4.6}), (\ref{4.8}), and (\ref{4.9}), we obtain the limit (\ref{3.1.3}).

Finally, we prove the case of the inverse local time clock. By the strong Markov property, we have
\begin{align}
\label{4.10}
h^B(c)\P_x(T_a\wedge T_b>\eta_u^c)&=\P_c(T_a\wedge T_b >\eta_u^c)\cdot h^B(c)\P_x(T_a\wedge T_b>T_c)
\end{align}
Since $L_{T_{a-c}\wedge T_{a-c}}^0$ has an exponential distribution with parameter $\frac{1}{h^C(a-c,b-c)}$ (see, the proof of Lemma 6.3 of \cite{TY}), we have
\begin{align}
\label{4.11}
\P_c(T_a\wedge T_b>\eta_u^c)=\P_0(L_{T_{a-c}\wedge T_{b-c}}^0>u)=e^{-\frac{u}{h^C(a-c,b-c)}}\to 1,
\end{align}
as $c\to \pm \infty.$ Therefore, by (\ref{4.10}) and (\ref{4.11}), we obtain the limit (\ref{3.1.4}).
\end{proof}

Next, we prove of the $L^1(\P_x)$-convergence.

\begin{thm}
\label{3.2}
The following assertions hold:
\begin{align}
\label{3.2.1}
\text{(Ex)}&\ \lim_{q\to 0+}r_q(0)\P_{X_t}(T_a\wedge T_b>\bm{e}_q)=\varphi_{a,b}^{(0)}(X_t)\ \text{in}\ L^1(\P_x),\\
\text{(OH)}&\ \lim_{c\to \pm\infty}h^B(c)\P_{X_t}(T_a\wedge T_b>T_c)=\varphi_{a,b}^{(\pm 1)}(X_t)\ \text{in}\ L^1(\P_x),\\
\text{(TH)}&\ \lim_{(c,d)\stackrel{(\gamma)}{\to} \infty}h^C(c,-d)\P_{X_t}(T_a\wedge T_b>T_c\wedge T_{-d})=\varphi_{a,b}^{(\gamma)}(X_t)\ \text{in}\ L^1(\P_x),\\
\text{(IL)}&\ \lim_{c\to \pm\infty}h^B(c)\P_{X_t}(T_a\wedge T_b>\eta_u^c)=\varphi_{a,b}^{(\pm 1)}(X_t)\ \text{in}\ L^1(\P_x).
\end{align}
\end{thm}
\begin{proof}
Using the limit 
\begin{align}
\lim_{q\to 0+}h_q(X_t)=h(X_t)\ \text{in}\ L^1(\P_x)
\end{align}
(see, Theorem 15.2 of \cite{Tukada}), the case of the exponential clock can be shown.

Next, using the limit 
\begin{align}
\lim_{c\to \pm \infty}h^B(c)\P_{X_t}(T_c<T_a)=h^{(\pm 1)}(X_t-a)\ \text{in}\ L^1(\P_x)
\end{align}
(see, (4.27) of \cite{IY}), the case of the one-point hitting time clock, the two-point hitting time clock, and the inverse local time clock can be shown.
\end{proof}

We give the proof of the main theorem for conditioning to avoid two points. 

\begin{proof}[The proof of Theorem \ref{t1.1}]
Note that by (\ref{0-27}), we have
\begin{align}
\lim_{|x|\to \infty}\frac{\varphi_{a,b}^{(0)}(x)}{x}=\lim_{|x|\to \infty}\left(\frac{h(x-a)}{x}-\frac{\P_x(T_b<T_a)h^{(\gamma)}(b-a)}{x}\right)=\frac{1}{m^2}.
\end{align}
Therefore, Theorem \ref{t1.1} holds by Theorems \ref{t2}, \ref{t3}, \ref{t4}, \ref{t5}, \ref{3.1}, and \ref{3.2}.
\end{proof}

\subsection{Limit measures}
\label{SS1}
We set $\mathscr{H}^{(\gamma)}_{a,b}:=\{x\in \R;\ \varphi^{(\gamma)}_{a,b}(x)>0\},$ and we define a probability measure
\begin{align}
\P_{x;a,b}^{(\gamma)}|_{\F_t}:=\frac{\varphi_{a,b}^{(\gamma)}(X_t)}{\varphi_{a,b}^{(\gamma)}(x)}1_{\{T_a\wedge T_b>t\}}\cdot \P_x|_{\F_t}\qquad \text{for}\ x\in \mathscr{H}^{(\gamma)}_{a,b}.
\end{align}
The measure $\P_{x;a,b}^{(\gamma)}$ can be well-defined on $\F_\infty:=\sigma(X_t,\ t\ge 0)$ (see, Theorem 9.1 of \cite{Yano}). By the strong Markov property of $\P_x$ and by the martingale property, we have
\begin{align*}
\label{c3}
\P_{x;a,b}^{(\gamma)}(T_{\R\setminus \mathscr{H}_{a,b}^{(\gamma)}}\le t)&=\P_x\left[1_{\{T_{\R\setminus \mathscr{H}_{a,b}^{(\gamma)}}\le t\}}\cdot \frac{\varphi_{a,b}^{(\gamma)}(X_t)}{\varphi_{a,b}^{(\gamma)}(x)}1_{\{T_a\wedge T_b>t\}}\right]\\
&=\P_x\left[1_{\{T_{\R\setminus \mathscr{H}_{a,b}^{(\gamma)}}\le t\}}\P_{X_{T_{\R\setminus \mathscr{H}_{a,b}^{(\gamma)}}}}\left[\frac{\varphi_{a,b}^{(\gamma)}(X_{t-s})}{\varphi_{a,b}^{(\gamma)}(x)}1_{\{T_a\wedge T_b>t-s\}}\right]\Big|_{s=T_{\R\setminus \mathscr{H}_{a,b}^{(\gamma)}}}\right]\\
&=\P_x\left[1_{\{T_{\R\setminus \mathscr{H}_{a,b}^{(\gamma)}}\le t\}}\frac{\varphi_{a,b}^{(\gamma)}(X_{{T}_{\R\setminus \mathscr{H}_{a,b}^{(\gamma)}}})}{\varphi_{a,b}^{(\gamma)}(x)}\right]=0,
\stepcounter{equation}\tag{\theequation}
\end{align*}
for any $t>0$. Thus, we have $\P_{x;a,b}^{(\gamma)}(T_{\R\setminus \mathscr{H}_{a,b}^{(\gamma)}}> t)=1$ for any $t>0.$ In particular, we have $\mathbb{P}_{x;a,b}^{(\gamma)}(T_a\wedge T_b>t)=1$ for any $t>0$. Therefore, we obtain
\begin{align}
\P_{x;a,b}^{(\gamma)}(T_a=T_b=\infty)=1.
\end{align}
Thus, the measure $\mathbb{P}_{x;a,b}^{(\gamma)}$ is absolutely continuous with respect to $\P_x$ on $\F_t$, but is singular to $\P_x$ on $\F_\infty$ since $\P_x(T_a\wedge T_b<\infty)=1.$

We set $\mathscr{H}_a^{(\gamma)}:=\{x\in \R;\ h^{(\gamma)}(x-a)>0\},$ and we define
\begin{align}
\P_{x;a}^{(\gamma)}|_{\F_t}:=\frac{h^{(\gamma)}(X_t-a)}{h^{(\gamma)}(x-a)}1_{\{T_a>t\}}\cdot \P_x|_{\F_t}\qquad \text{for}\ x\in \mathscr{H}^{(\gamma)}_{a}.
\end{align}
We know that $\mathscr{H}_{a}^{(\gamma)}=\R\setminus \{a\},\ (a,\infty)$, or $(-\infty,a)$ (see, p.12 of \cite{Takeda}). Note that
\begin{align}
\mathscr{H}_{a,b}^{(\gamma)}\subset \mathscr{H}_a^{(\gamma)}\cap \mathscr{H}_b^{(\gamma)}.
\end{align}

Although The proof of the following proposition is parallel to Theorem 1.4 of Takeda \cite{Takeda}, we give the proof for completeness of this paper.
\begin{prop}
\label{c6}
For $-1\le \gamma\le 1$ and $x\in \mathscr{H}_{a,b}^{(\gamma)}$, the process $((X_t),\P_{x;a,b}^{(\gamma)})$ is transient.
\end{prop}

\begin{proof}
For $0<s<t$ and a non-negative bounded $\F_s$-measurable functional $F_s$, we have
\begin{align*}
\P_{x;a,b}^{(\gamma)}\left[\frac{1}{\varphi_{a,b}^{(\gamma)}(X_t)}\cdot F_s\right]&=\frac{1}{\varphi_{a,b}^{(\gamma)}(x)}\P_x[F_s;\ T_a\wedge T_b>t]\\
&\le \frac{1}{\varphi_{a,b}^{(\gamma)}(x)}\P_x[F_s;\ T_a\wedge T_b>s]=\P_{x;a,b}^{(\gamma)}\left[\frac{1}{\varphi_{a,b}^{(\gamma)}(X_s)}\cdot F_s\right].
\stepcounter{equation}\tag{\theequation}
\end{align*}
Thus, $(\frac{1}{\varphi_{a,b}^{(\gamma)}(X_t)})_{t\ge 0}$ is a non-negative $\P_{x;a,b}^{(\gamma)}$-supermartingale. By the martingale convergence theorem, $\lim_{t\to \infty}\frac{1}{\varphi_{a,b}^{(\gamma)}(X_t)}$ exists $\P_{x;a,b}^{(\gamma)}$-a.s. By Fatou's lemma and recurrence of $((X_t),\P_x)$, we have
\begin{align}
\P_{x;a,b}^{(\gamma)}\left[\lim_{t\to \infty}\frac{1}{\varphi_{a,b}^{(\gamma)}(X_t)}\right]\le \varliminf_{t\to \infty}\P_{x;a,b}^{(\gamma)}\left[\frac{1}{\varphi_{a,b}^{(\gamma)}(X_t)}\right]=\frac{1}{\varphi_{a,b}^{(\gamma)}}\varliminf_{t\to \infty}\P_x(T_a\wedge T_b>t)=0.
\end{align}
Thus, we have $\lim_{t\to \infty}\frac{1}{\varphi_{a,b}^{(\gamma)}(X_t)}=0$ $\P_{x;a,b}^{(\gamma)}$-a.s., and it implies $\lim_{t\to \infty}|X_t|=\infty$ $\P_{x;a,b}^{(\gamma)}$-a.s. Therefore, the process $((X_t),\P_{x;a,b}^{(\gamma)})$ is transient.
\end{proof}

We define the measures on $\F_\infty$ as follows:
\begin{align}
\mathscr{P}_{x;a}^{(\gamma)}&:=h^{(\gamma)}(x-a)\cdot \P_{x;a}^{(\gamma)}\qquad \ \mathrm{for}\ x\in\mathscr{H}_{a}^{(\gamma)},\\
\mathscr{P}_{x;a,b}^{(\gamma)}&:=\varphi_{a,b}^{(\gamma)}(x)\cdot \P_{x;a,b}^{(\gamma)}\qquad \mathrm{for}\ x\in\mathscr{H}_{a,b}^{(\gamma)}.
\end{align}
Note that $\mathscr{P}_{x;a}^{(\gamma)}$ and $\mathscr{P}_{x;a,b}^{(\gamma)}$ are bounded measures.

\begin{prop}
Let $-1\le \gamma\le 1$ and $x\in \mathscr{H}_{a,b}^{(\gamma)}$. Then, it holds that
\begin{align}
\mathscr{P}_{x;a,b}^{(\gamma)}=1_{\{T_b=\infty\}}\cdot \mathscr{P}_{x;a}^{(\gamma)}.
\end{align}
\end{prop}

\begin{proof}
Since $((X_t),\P_{x;a,b}^{(\gamma)})$ is transient, we have $\lim_{t\to \infty}|X_t|=\infty,\ \P_{x;a,b}^{(\gamma)}$-a.s. Thus, we have
\begin{align}
\lim_{t\to \infty}\frac{\varphi^{(\gamma)}_{a,b}(X_t)}{h^{(\gamma)}(X_t-a)}=1,\qquad \P_{x;a,b}^{(\gamma)}\text{-a.s.}
\end{align}
Therefore, we apply to Theorem 4.1 of \cite{Yano} as $\Gamma_t=1_{\{T_a\wedge T_b>t\}}$ and $\mathscr{E}_t=1_{\{T_a>t\}}$, then the assertion holds.
\end{proof}

Consequently, we obtain
\begin{align}
\P_{x;a,b}^{(\gamma)}=\frac{h^{(\gamma)}(x-a)}{\varphi_{a,b}^{(\gamma)}(x)}1_{\{T_b=\infty\}}\cdot \mathbb{P}_{x;a}^{(\gamma)}.
\end{align}
This implies 
\begin{align}
\varphi_{a,b}^{(\gamma)}(x)=h^{(\gamma)}(x-a)\P_{x;a}^{(\gamma)}(T_b=\infty).
\end{align}
In particular, we have
\begin{align}
\P_{x;a}^{(\gamma)}(\cdot |\ T_b=\infty)=\P_{x;a,b}^{(\gamma)}(\cdot).
\end{align}
Therefore, $\P_{x;a,b}^{(\gamma)}$ is absolutely continuous with respect to $\P_{x;a}^{(\gamma)}.$

\subsection{Examples}
\subsubsection*{Brownian motion}
Let $(X_t)$ be a standard Brownian motion. Then, we know that $h^{(\gamma)}(x)=|x|+\gamma x$ (see, e.g., Example 3.9 of \cite{TY}). Note that
\begin{align}
\P_x(T_a<T_b)=\frac{b-x}{b-a}\qquad \text{for}\ a<x<b
\end{align}
(see, e.g., Theorem 7.5.3 of \cite{Durret}).
\begin{enumerate}
\item If $x<a,$ then we have
\begin{align}
\varphi^{(\gamma)}_{a,b}(x)&=a-x+\gamma(x-a)=(1-\gamma)(a-x).
\end{align}
Thus, we obtain $\P_{x;a,b}^{(\gamma)}=\P_{x;a}^{(\gamma)}$.
\item If $a<x<b$, then we have
\begin{align}
\varphi^{(\gamma)}_{a,b}(x)&=x-a+\gamma(x-a)-\frac{x-a}{b-a}(b-a+\gamma(b-a))=0.
\end{align}
\item If $b<x$, then we have
\begin{align}
\varphi^{(\gamma)}_{a,b}(x)&=x-a+\gamma(x-a)-(b-a+\gamma(b-a))=(1+\gamma)(x-b).
\end{align}
Thus, we obtain $\P_{x;a,b}^{(\gamma)}=\P_{b;x}^{(\gamma)}$.
\end{enumerate}
Therefore, we obtain $\mathscr{H}_{a,b}^{(\gamma)}=(-\infty,a)\cup (b,\infty).$

\subsubsection*{Stable process}
Let $(X_t)$ be a strictly stable process of index $\alpha\in (1,2).$ Then, we know that
\begin{align}
h^{(\gamma)}(x)=\frac{1}{K(\alpha)}(1-\beta\ \mathrm{sgn}(x))|x|^{\alpha-1}
\end{align}
(see, Section 5 of \cite{Yano2}). Note that $h^{(\gamma)}$ is independent of $\gamma.$ By Lemma 3.5 of \cite{TY}, we have
\begin{align*}
\P_x(T_b<T_a)=\frac{(1+\beta)(b-a)^{\alpha-1}+(1-\beta\ \mathrm{sgn}(x-a))|x-a|^{\alpha-1}-(1-\beta\ \mathrm{sgn}(x-b))|x-b|^{\alpha-1}}{2(b-a)^{\alpha-1}}
\stepcounter{equation}\tag{\theequation}
\end{align*}
for $x\in \R.$ 
\begin{enumerate}
\item If $x<a$, then we have
\begin{align}
\varphi_{a,b}^{(\gamma)}(x)=\frac{1}{K(\alpha)}\left\{\frac{(1+\beta)^2}{2}(a-x)^{\alpha-1}+\frac{(1+\beta)(1-\beta)}{2}(b-x)^{\alpha-1}-\frac{(1+\beta)(1-\beta)}{2}(b-a)^{\alpha-1}\right\}.
\end{align}
\item If $a<x<b$, then
\begin{align}
\varphi_{a,b}^{(\gamma)}(x)=\frac{1}{K(\alpha)}\frac{(1+\beta)(1-\beta)}{2}\left\{(x-a)^{\alpha-1}+(b-x)^{\alpha-1}-(b-a)^{\alpha-1}\right\}.
\end{align}
\item If $b<x$, then
\begin{align}
\varphi_{a,b}^{(\gamma)}(x)=\frac{1}{K(\alpha)}\left\{\frac{(1+\beta)(1-\beta)}{2}(x-a)^{\alpha-1}+\frac{(1-\beta)^2}{2}(x-b)^{\alpha-1}-\frac{(1+\beta)(1-\beta)}{2}(b-a)^{\alpha-1}\right\}.
\end{align}
\end{enumerate}
Therefore, we obtain $\mathscr{H}_{a,b}^{(\gamma)}=\R\setminus \{a,b\}$ when $\beta\neq \pm 1$, $\mathscr{H}_{a,b}^{(\gamma)}=(-\infty,a)$ when $\beta=1$, and $\mathscr{H}_{a,b}^{(\gamma)}=(b,\infty)$ when $\beta=-1.$

\subsubsection*{Spectrally negative \Levy\ process}
Let $(X_t)$ be a  recurrent spectrally negative \Levy\ process. Then, we know that 
\begin{align}
h^{(\gamma)}(x)=W(x)+\frac{\gamma-1}{m^2}x,
\end{align}
where $W(x)$ is the scale function of $(X_t)$ (see, Example 28 of \cite{Pa}). By Lemma 3.5 of \cite{TY}, we have
\begin{align}
\P_x(T_b<T_a)=\frac{W(x-a)-W(x-b)}{W(b-a)}.
\end{align}
\begin{enumerate}
\item If $x<a$, then we have
\begin{align}
\varphi^{(\gamma)}_{a,b}(x)=\frac{\gamma-1}{m^2}(x-a).
\end{align}
\item If $a<x<b$, then we have
\begin{align}
\varphi_{a,b}^{(\gamma)}(x)=\frac{\gamma-1}{m^2}\left\{(x-a)-\frac{W(x-a)}{W(b-a)}(b-a)\right\}.
\end{align}
\item If $b<x$, then we have
\begin{align}
\varphi_{a,b}^{(\gamma)}(x)=W(x-b)+\frac{\gamma-1}{m^2}\left\{(x-a)-\frac{W(x-a)-W(x-b)}{W(b-a)}(b-a)\right\}.
\end{align}
\end{enumerate}
Therefore, we obtain $\mathscr{H}_{a,b}^{(\gamma)}=\R\setminus \{a,b\}$ when $m^2<\infty$, and $\mathscr{H}_{a,b}^{(\gamma)}=(b,\infty)$ when $m^2=\infty.$

%%%%%%%%%%%%%%%%%%%%%%%%%%%%%%%%%%%%%%%%%%%%%%%%%%%%%%%%%%%%%%%%%%%%%%%%%%%%%%%%%%%%%%%%%%%%%%%%%%%%%%%%%%%%%%%%%%%%%%%%%%%%%%%%%%%%%%%%%%%%%%%%%%%%%%%%%%%%%%%%%%%%%%%%%%%%%%%%%%%%%%%%%%%%%%%%%%%%%%%%%%%%%%%%%%%%%%%%%%%%%%%%%%%%%%%%%%%%%%%%%%%%%%%%%%%%%

\section{Conditioning to avoid $n$-points}
\label{S4}
\subsection{The case of $x\neq a_n$}
For a sequence $a_1,a_2,...$ of distinct points of $\R$, we set $A_n:=\{a_1,...,a_n\}$ for $n=2,3,...$. We define for $n=2,3,...$,
\begin{align}
\label{d1}
\varphi_{A_n}^{(\gamma)}(x):=h^{(\gamma)}(x-a_n)\mathbb{P}^{(\gamma)}_{x;a_n}(T_{A_{n-1}}=\infty).
\end{align}
Note that $\varphi_{A_2}^{(\gamma)}(x)$ is already defined in (\ref{c1}). For any $t>0$, we have by the optional sampling theorem,
\begin{align*}
\P_{x;a_n}^{(\gamma)}(T_{A_{n-1}}<t)&=\P_x\left[1_{\{T_{A_{n-1}}<t\}}\cdot \frac{h^{(\gamma)}(X_t-a_n)}{h^{(\gamma)}(x-a_n)}1_{\{T_{a_n}>t\}}\right]\\
&=\P_x\left[1_{\{T_{A_{n-1}}<t\}}\cdot \P_x\left[\frac{h^{(\gamma)}(X_t-a_n)}{h^{(\gamma)}(x-a_n)}1_{\{T_{a_n}>t\}}\Big|\F_{T_{A_{n-1}}}\right]\right]\\
&=\P_x\left[1_{\{T_{A_{n-1}}<t\}}\cdot \frac{h^{(\gamma)}(X_{T_{A_{n-1}}}-a_n)}{h^{(\gamma)}(x-a_n)}1_{\{T_{a_n}>T_{A_{n-1}}\}}\right].
\stepcounter{equation}\tag{\theequation}
\end{align*}
Letting $t\to \infty,$ we have
\begin{align*}
\P_{x;a_n}^{(\gamma)}(T_{A_{n-1}}<\infty)&=\P_x\left[\frac{h^{(\gamma)}(X_{T_{A_{n-1}}}-a_n)}{h^{(\gamma)}(x-a_n)}1_{\{T_{a_n}>T_{A_{n-1}}\}}\right]\\
&=\frac{1}{h^{(\gamma)}(x-a_n)}\sum_{k=1}^{n-1} \P_x\left[h^{(\gamma)}(X_{T_{A_{n-1}}}-a_n)1_{\{T_{a_n}>T_{A_{n-1}}=T_{a_k}\}}\right]\\
&=\frac{1}{h^{(\gamma)}(x-a_n)}\sum_{k=1}^{n-1} h^{(\gamma)}(a_k-a_n)\P_x(T_{a_k}= T_{A_n}).
\stepcounter{equation}\tag{\theequation}
\end{align*}
Therefore, we obtain
\begin{align*}
\varphi_{A_n}^{(\gamma)}(x)&=h^{(\gamma)}(x-a_n)-\sum_{k=1}^{n-1} h^{(\gamma)}(a_k-a_n)\P_x(T_{a_k}= T_{A_n})\\
&=h^{(\gamma)}(x-a_n)-\P_x[h^{(\gamma)}(X_{T_{A_n}}-a_n)].
\stepcounter{equation}\tag{\theequation}
\end{align*}
Moreover, by a simple calculation, we have
\begin{align}
\varphi_{A_n}^{(\gamma)}(x)=\varphi_{A_{n-1}}^{(\gamma)}(x)-\varphi_{A_{n-1}}^{(\gamma)}(a_n)\P_{x}(T_{a_n}< T_{A_{n-1}}).
\end{align}
By the induction and (6.2) of \cite{TY}, we can express $\P_x(T_{a_k}= T_{A_n})$ only in terms of $h$. In particular, we can express $\varphi_{A_n}^{(\gamma)}(x)$ as a linear combination of $h^{(\gamma)}(x-a_1),...,\ h^{(\gamma)}(x-a_n).$

First, we consider the a.s. limits.

\begin{prop}
\label{dp1}
The following assertions hold:
\begin{align}
\text{(Ex)}&\ \lim_{q\to 0+}r_q(0)\P_x(T_{A_n}>\bm{e}_q)=\varphi_{A_n}^{(0)}(x),\\
\text{(OH)}&\ \lim_{c\to \pm\infty}h^B(c)\P_x(T_{A_n}>T_c)=\varphi_{A_n}^{(\pm 1)}(x),\\
\text{(TH)}&\ \lim_{(c,d)\stackrel{(\gamma)}{\to} \infty}h^C(c,-d)\P_x(T_{A_n}>T_c\wedge T_{-d})=\varphi_{A_n}^{(\gamma)}(x),\\
\text{(IL)}&\ \lim_{c\to \pm\infty}h^B(c)\P_x(T_{A_n}>\eta_u^c)=\varphi_{A_n}^{(\pm 1)}(x).
\end{align}
\end{prop}
\begin{proof}
First, we show the case of the exponential clock by the induction. The case $n=2$ is already shown in (\ref{3.1.1}). For general $n\ge 3$, we have by the strong Markov property and by the hypothesis of induction,
\begin{align*}
r_q(0)\P_x(T_{A_n}>\bm{e}_q)&=r_q(0)\P_x(T_{A_{n-1}}>\bm{e}_q)-r_q(0)\P_x(T_{A_{n-1}}>\bm{e}_q> T_{a_{n}})\\
&=r_q(0)\P_x(T_{A_{n-1}}>\bm{e}_q)-r_q(0)\P_{a_n}(T_{A_{n-1}}>\bm{e}_q)\P_x(T_{a_n}< T_{A_{n-1}}\wedge \bm{e}_q)\\
&\to \varphi_{A_{n-1}}^{(0)}(x)-\varphi_{A_{n-1}}^{(0)}(a_n)\P_{x}(T_{a_n}< T_{A_{n-1}})=\varphi_{A_n}^{(0)}(x),
\stepcounter{equation}\tag{\theequation}
\end{align*}
as $q\to 0+$.

The rest of the proof is the same as the case of the exponential clock, so we omit it.
\end{proof}

Next, we consider the $L^1(\P_x)$-limits.

\begin{prop}
\label{dp2}
The following assertions hold:
\begin{align}
\text{(Ex)}&\ \lim_{q\to 0+}r_q(0)\P_{X_t}(T_{A_n}>\bm{e}_q)=\varphi_{A_n}^{(0)}(X_t)\ \text{in}\ L^1(\P_x),\\
\text{(OH)}&\ \lim_{c\to \pm\infty}h^B(c)\P_{X_t}(T_{A_n}>T_c)=\varphi_{A_n}^{(\pm 1)}(X_t)\ \text{in}\ L^1(\P_x),\\
\text{(TH)}&\ \lim_{(c,d)\stackrel{(\gamma)}{\to} \infty}h^C(c,-d)\P_{X_t}(T_{A_n}>T_c\wedge T_{-d})=\varphi_{A_n}^{(\gamma)}(X_t)\ \text{in}\ L^1(\P_x),\\
\text{(IL)}&\ \lim_{c\to \pm\infty}h^B(c)\P_{X_t}(T_{A_n}>\eta_u^c)=\varphi_{A_n}^{(\pm 1)}(X_t)\ \text{in}\ L^1(\P_x).
\end{align}
\end{prop}
\begin{proof}
First, we show the case of the exponential clock by the induction. The case $n=2$ is already shown in (\ref{3.2.1}). For general $n\ge 3$, we have by the strong Markov property
\begin{align*}
&|r_q(0)\P_{X_t}(T_{A_n}>\bm{e}_q)-\varphi_{A_n}^{(0)}(X_t)|\\
&=|r_q(0)\P_{X_t}(T_{A_{n-1}}>\bm{e}_q)-r_q(0)\P_{a_n}(T_{A_{n-1}}>\bm{e}_q)\P_{X_t}(T_{a_n}< T_{A_{n-1}}\wedge \bm{e}_q)\\
&\qquad -\varphi_{A_{n-1}}^{(0)}(X_t)+\P_{X_t}(T_{a_n}< T_{A_{n-1}})\varphi_{A_{n-1}}^{(0)}(a_n)|\\
&\le |r_q(0)\P_{X_t}(T_{A_{n-1}}>\bm{e}_q)-\varphi_{A_{n-1}}^{(0)}(X_t)|\\
&\qquad +|\P_{X_t}(T_{a_n}< T_{A_{n-1}})\varphi_{A_{n-1}}^{(0)}(a_n)-\P_{X_t}(T_{a_n}< T_{A_{n-1}}\wedge \bm{e}_q)\varphi_{A_{n-1}}^{(0)}(a_n)|\\
&\qquad +|\P_{X_t}(T_{a_n}< T_{A_{n-1}}\wedge \bm{e}_q)\varphi_{A_{n-1}}^{(0)}(a_n)-r_q(0)\P_{a_n}(T_{A_{n-1}}>\bm{e}_q)\P_{X_t}(T_{a_n}< T_{A_{n-1}}\wedge \bm{e}_q)|\\
&= |r_q(0)\P_{X_t}(T_{A_{n-1}}>\bm{e}_q)-\varphi_{A_{n-1}}^{(0)}(X_t)|\\
&\qquad +\varphi_{A_{n-1}}^{(0)}(a_n)|\P_{X_t}(T_{a_n}< T_{A_{n-1}})-\P_{X_t}(T_{a_n}< T_{A_{n-1}}\wedge \bm{e}_q)|\\
&\qquad +\P_{X_t}(T_{a_n}< T_{A_{n-1}}\wedge \bm{e}_q)|\varphi_{A_{n-1}}^{(0)}(a_n)-r_q(0)\P_{a_n}(T_{A_{n-1}}>\bm{e}_q)|.
\stepcounter{equation}\tag{\theequation}
\end{align*}
The first term converges to $0$ by the hypothesis of induction, the second term by the dominated convergence theorem, and the third term by Proposition \ref{dp1}, respectively. 

The rest of the proof is the same as the case of the exponential clock, so we omit it.
\end{proof}

We give the proof of the main theorem for conditioning to avoid $n$ points. 

\begin{proof}[The proof of Theorem \ref{t1.2}]
Note that by (\ref{0-27}), we have
\begin{align}
\lim_{|x|\to \infty}\frac{\varphi_{A_n}^{(0)}(x)}{x}=\lim_{|x|\to \infty}\left(\frac{h(x-a_n)}{x}-\sum_{k=1}^{n-1}\frac{\P_x(T_{a_k}<T_{A_n})h^{(\gamma)}(a_k-a_n)}{x}\right)=\frac{1}{m^2}.
\end{align}
Therefore, Theorem \ref{t1.2} holds by Theorems \ref{t2}, \ref{t3}, \ref{t4}, \ref{t5}, \ref{dp1}, and \ref{dp2}.
\end{proof}

From Theorem \ref{t1.2}, we obtain for any $k=1,...,n$,
\begin{align*}
\varphi_{A_n}^{(\gamma)}(x)&=h^{(\gamma)}(x-a_k)\mathbb{P}^{(\gamma)}_{x;a_k}(T_{A_n}=\infty)\\
&=h^{(\gamma)}(x-a_k)-\sum_{i=1}^{n} h^{(\gamma)}(a_i-a_k)\P_x(T_{a_i}= T_{A_n})\\
&=h^{(\gamma)}(x-a_k)-\P_x[h^{(\gamma)}(X_{T_{A_n}}-a_k)].
\stepcounter{equation}\tag{\theequation}
\end{align*}

We define
\begin{align}
\P_{x;A_n}^{(\gamma)}|_{\F_t}&:=1_{\{T_{A_n}>t\}}\frac{\varphi_{A_n}^{(\gamma)}(X_t)}{\varphi_{A_n}^{(\gamma)}(x)}\cdot \P_x|_{\F_t}
\end{align}
for $x\in \{x\in \R;\ \varphi_{A_n}^{(\gamma)}(x)> 0\}$. In the same way as (\ref{c3}), we obtain
\begin{align}
  \P_{x;A_n}^{(\gamma)}(T_{A_n}=\infty)=1.
\end{align}
The measure $\P_{x;A_n}^{(\gamma)}$ is absolutely continuous with respect to $\P_x$ on $\F_t$, but is singular to $\P_x$ on $\F_\infty$. We obtain by the same way as Subsection \ref{SS1},
\begin{align}
\P_{x;a_n}^{(\gamma)}(\cdot |\ T_{A_{n-1}}=\infty)=\P_{x;A_n}^{(\gamma)}(\cdot).
\end{align}

\subsection{The case of $x=a_n$}
We define the measure
\begin{align}
\label{d2-1}
\P_{a_n;A_n}^{(\gamma)}|_{\F_t}:=\varphi_{A_n}^{(\gamma)}(X_t)1_{\{T_{A_n}>t\}}\cdot n^{a_n}|_{\F_t},
\end{align}
where $n^{a_n}$ is the excursion measure away from $a_n$ for the process $(X,\P_x)$. We characterize this measure as follows:

\begin{thm}
Let $F_t$ be a bounded $\F_t$-measurable functional. Then, the following assertions hold:
\begin{align}
\text{(Ex)}&\ \lim_{q\to 0+}\P_{a_n}\left[F_t\circ k_{\bm{e}_q-g_{\bm{e}_q}^{a_n}}\circ\theta_{g_{\bm{e}_q}^{a_n}},\ T_{A_{n-1}}\circ \theta_{g_{\bm{e}_q}^{a_n}}>\bm{e}_q-g_{\bm{e}_q}^{a_n}\right]=\P_{a_n;A_n}^{(0)}[F_t],\\
\text{(OH)}&\ \lim_{c\to \pm \infty}\P_{a_n}\left[F_t\circ k_{T_c-g_{T_c}^{a_n}}\circ\theta_{g_{T_c}^{a_n}},\ T_{A_{n-1}}\circ \theta_{g_{T_c}^{a_n}}>T_c-g_{T_c}^{a_n}\right]=\P_{a_n;A_n}^{(\pm 1)}[F_t],\\
\text{(TH)}&\ \lim_{(c,d)\stackrel{(\gamma)}{\to}\infty}\P_{a_n}\left[F_t\circ k_{T_c\wedge T_{-d}-g_{T_c\wedge T_{-d}}^{a_n}}\circ\theta_{g_{T_c\wedge T_{-d}}^{a_n}},\ T_{A_{n-1}}\circ \theta_{g_{T_c\wedge T_{-d}}^{a_n}}>T_c\wedge T_{-d}-g_{T_c\wedge T_{-d}}^{a_n}\right]=\P_{a_n;A_n}^{(\gamma)}[F_t],
\end{align}
where $\theta$ is the shift operator, $k$ is the killing operator, and $g_s^{a}$ is the last hitting time of a point $a$ up to time $s$, i.e., for $t>0$ and \cadlag\ paths $\omega$,
\begin{align}
k_{t-g_t}\circ \theta_{g_t}\omega(s)=\begin{cases}
\omega(g_t+s)&\text{if}\ 0\le s<t-g_t,\\
\text{a cemetery state} &\text{if}\ s\ge t-g_t.
\end{cases}
\end{align}
\end{thm}

\begin{proof}
For $s>0,$ we define $d_s^{a_n}:=\inf\{u>s;\ X_u=a_n\}$ and $G^{a_n}:=\{g_s^{a_n};\ g_s^{a_n}\neq d_s^{a_n},\ s>0\}$. First, we show the case of the exponential clock. For any $q>0$, we have
\begin{align*}
\P_{a_n}&\left[F_t\circ k_{\bm{e}_q-g_{\bm{e}_q}^{a_n}}\circ\theta_{g_{\bm{e}_q}^{a_n}},\ T_{A_{n-1}}\circ \theta_{g_{\bm{e}_q}^{a_n}}>\bm{e}_q-g_{\bm{e}_q}^{a_n}\right]\\
&=\P_{a_n}\left[\int_0^\infty qe^{-qu}F_t\circ k_{u-g_{u}^{a_n}}\circ\theta_{g_{u}^{a_n}}1_{\{T_{A_{n-1}}\circ \theta_{g_{u}^{a_n}}>u-g_{u}^{a_n}\}}du\right]\\
&=\P_{a_n}\left[\sum_{s\in G^{a_n}}\int_s^{d_s^{a_n}}qe^{-qu}F_t\circ k_{u-s}\circ\theta_{s}1_{\{T_{A_{n-1}}\circ \theta_s>u-s\}}du\right].
\stepcounter{equation}\tag{\theequation}
\end{align*}
By the compensation formula (see, e.g., Corollary IV.11 of \cite{Ber}), we have
\begin{align*}
\P_{a_n}&\left[\sum_{s\in G^{a_n}}\int_s^{d_s^{a_n}}qe^{-qu}F_t\circ k_{u-s}\circ\theta_{s}1_{\{T_{A_{n-1}}\circ \theta_s>u-s\}}du\right]\\
&=\P_{a_n}\left[\int_0^\infty e^{-qs}dL_s^{a_n}\right]n^{a_n}\left[\int_0^{T_{a_n}}qe^{-qu}F_t1_{\{u>t\}}1_{\{T_{A_{n-1}}>u\}}du\right]\\
&=r_q(0)n^{a_n}[F_t,\ t<\bm{e}_q<T_{A_n}]\\
&=r_q(0)e^{-qt}n^{a_n}\left[F_t\cdot \P_{X_t}(\bm{e}_q<T_{A_n}),\ t<T_{A_n}\right].
\stepcounter{equation}\tag{\theequation}
\end{align*}
Since 
\begin{align}
r_q(0)\P_{X_t}(\bm{e}_q<T_{A_n})\le r_q(0)\P_{X_t}(\bm{e}_q<T_{a_1}),
\end{align}
which is integrable by the proof of Theorem 1.3 of Takeda \cite{Takeda}, we have
\begin{align*}
\lim_{q\to 0+}&\P_{a_n}\left[F_t\circ k_{\bm{e}_q-g_{\bm{e}_q}^{a_n}}\circ\theta_{g_{\bm{e}_q}^{a_n}},\ T_{A_{n-1}}\circ \theta_{g_{\bm{e}_q}^{a_n}}>\bm{e}_q-g_{\bm{e}_q}^{a_n}\right]\\
&=\lim_{q\to 0+}r_q(0)e^{-qt}n^{a_n}\left[F_t\cdot \P_{X_t}(\bm{e}_q<T_{A_n}),\ t<T_{A_n}\right]\\
&=n^{a_n}\left[F_t\cdot \varphi_{A_n}^{(\gamma)}(X_t),\ T_{A_n}>t\right]=\P_{a_n;A_n}^{(0)}[F_t],
\stepcounter{equation}\tag{\theequation}
\end{align*}
by the dominated convergence theorem and Theorem \ref{dp1}. 

We omit the rest of proof, since it is similar to that of Theorem 1.3 of Takeda \cite{Takeda}.
\end{proof}

%%%%%%%%%%%%%%%%%%%%%%%%%%%%%%%%%%%%%%%%%%%%%%%%%%%%%%%%%%%%%%%%%%%%%%%%%%%%%%%%%%%%%%%%%%%%%%%%%%%%%%%%%%%%%%%%%%%%%%%%%%%%%%%%%%%%%%%%%%%%%%%%%%%%%%%%%%%%%%%%%%%%%%%%%%%%%%%%%%%%%%%%%%%%%%%%%%%%%%%%%%%%%%%%%%%%%%%%%%%%%%%%%%%%%%%%%%%%%%%%%%%%%%%%%%%%%%%%%%%%%%%%%%%%%%%%%%%%%

\section{Conditioning to avoid bounded $F_\sigma$-sets}
\label{S5}
Let $A$ be a $F_\sigma$-set. Note that the first hitting time $T_A$ is a stopping time. We denote $(X_t^A)_{t\ge 0}$ by the process which killed by hitting $A$, and denote $p_t^A$ by its transition density, i.e., for $x,y\in \R$, we have
\begin{align}
p_t^A(x,y)=p_t(y-x)-\P_x[p_{t-T_A}(y-X_{T_A});\ t>T_A],
\end{align}
where $p_t$ is the transition density of $((X_t),\P_0)$. For $q>0$, we denote $r_q^A$ by $q$-resolvent density of the killed process $(X_t^A)$, i.e.,
\begin{align}
r_q^A(x,y):=\int_0^\infty e^{-qt}p_t^A(x,y)dt.
\end{align}
Note that $r_q^A(x,y)$ is continuous in $y$, since we have by Fubini's theorem,
\begin{align*}
r_q^A(x,y)&=\int_0^\infty e^{-qt}p_t(y-x)-\int_0^\infty e^{-qt}\P_x[p_{t-T_A}(y-X_{T_A});\ t>T_A]dt\\
&=r_q(y-x)-\P_x\left[e^{-qT_A}r_q(y-X_{T_A})\right].
\end{align*}

\begin{lem}
\label{el1}
Let $A$ is a bounded $F_\sigma$-set which contains $0$. For any $q>0$ and $x\in \R$, it holds that $r_q^A(x,0)=0.$
\end{lem}

\begin{proof}
We define the $q$-resolvent operator $R_q$ by
\begin{align}
R^q f(x):=\P_x\left[\int_0^\infty e^{-qt}f(X_t)dt\right]
\end{align}
for a non-negative bounded Borel function $f.$ Then, we have for a stopping time $T$,
\begin{align}
\P_x[R_qf(X_T)\cdot e^{-qT}]=\P_x\left[\int_T^\infty e^{-qt}f(X_t)dt\right].
\end{align}
Indeed, by the strong Markov property, we have
\begin{align*}
\P_x[R_qf(X_T)\cdot e^{-qT}]&=\P_x\left[e^{-qT}\P_{X_T}\left[\int_0^\infty e^{-qt}f(X_{t})dt\right]\right]\\
&=\P_x\left[\int_0^\infty e^{-q(T+t)}f(X_{T+t})dt\right]\\
&=\P_x\left[\int_T^\infty e^{-qt}f(X_t)dt\right].
\stepcounter{equation}\tag{\theequation}
\end{align*}
Since $x\mapsto r_q(-x)$ is a $q$-excessive function, there exists a sequence of non-negative bounded Borel functions $(f_n)$ such that $R_qf_n(x)\nearrow  r_q(-x)$ for $x\in \R$ as $n\to \infty$ (see, e.g., Proposition 41.5 and Theorem 41.16 of \cite{Sato}). Since $T_A\le T_0$, we have
\begin{align*}
\P_x[R_qf_n(X_{T_A})\cdot e^{-q{T_A}}]&=\P_x\left[\int_{T_A}^\infty e^{-qt}f_n(X_t)dt\right]\\
&\ge \P_x\left[\int_{T_0}^\infty e^{-qt}f_n(X_t)dt\right]=\P_x[R_qf_n(X_{T_0})\cdot e^{-q{T_0}}].
\stepcounter{equation}\tag{\theequation}
\end{align*}
Thus, by the monotone convergence theorem, we have
\begin{align}
\P_x[r_q(-X_{T_A}) e^{-q{T_A}}]\ge \P_x[r_q(-X_{T_0}) e^{-q{T_0}}].
\end{align}
Therefore, we obtain
\begin{align*}
0\le r_q^A(x,0)&=r_q(-x)-\P_x[r_q(-X_{T_A})e^{-qT_A}]\\
&\le r_q(-x)-\P_x[r_q(-X_{T_0}) e^{-q{T_0}}]\\
&=r_q(-x)-r_q(0)\cdot \frac{r_q(-x)}{r_q(0)}=0.
\stepcounter{equation}\tag{\theequation}
\end{align*}
The proof is complete.
\end{proof}

The proof of the following proposition is inspired by Proposition 3.7 of Grzywny-Ryznar \cite{GR}.
\begin{prop}
\label{ep1}
Let $A$ is a bounded $F_\sigma$-set which contains $0$. For $x\in\R$, it holds that
\begin{align}
\lim_{q\to 0+}r_q(0)\P_x(T_A>\bm{e}_q)=h(x)-\P_x[h(X_{T_A})]=:\varphi_{A}(x).
\end{align}
\end{prop}

\begin{proof}
By Fubini's theorem and Lemma \ref{el1}, we have
\begin{align*}
\label{e1}
r_q(0)\P_x(T_A>\bm{e}_q)&=r_q(0)(1-\P_x[e^{-qT_A}])\\
&=r_q(0)-r_q(-x)-\P_x[r_q(0)e^{-qT_A}-r_q(-X_{T_A})e^{-qT_A}]\\
&\qquad +r_q(-x)-\P_x[r_q(-X_{T_A})e^{-qT_A}]\\
&=h_q(x)-\P_x[h_q(X_{T_A})e^{-qT_A}]+r_q^{A}(x,0)\\
&=h_q(x)-\P_x[h_q(X_{T_A})e^{-qT_A}].
\stepcounter{equation}\tag{\theequation}
\end{align*}
Since $A$ is a bounded set and $0\le h_q(x)\le h^B(x)$ which is continuous on $\R$ (see, e.g., (6.19) of \cite{Takeda}), we obtain by the dominated convergence theorem,
\begin{align}
r_q(0)\P_x(T_A>\bm{e}_q)\to h(x)-\P_x[h(X_{T_A})]
\end{align}
as $q\to 0+$. 
\end{proof}

Next, we prove of the $L^1(\P_x)$-convergence.

\begin{prop}
\label{ep2}
Let $A$ is a bounded set which contains $0$. For $t>0$, it holds that
\begin{align}
\label{e2}
\lim_{q\to 0+}r_q(0)\P_{X_t}(T_A>\bm{e}_q)=\varphi_A(X_t)\ \text{in}\ L^1(\P_x).
\end{align}
\end{prop}
\begin{proof}
Since $h^B$ is continuous on $\R$ and $A$ is a bounded set, there exists $M>0$ such that $0\le h^B(X_{T_A})\le M$. Note that $0\le h_q(x)\le h^B(x)$ (see, e.g., (6.9) of \cite{Takeda}). By (\ref{e1}), we have
\begin{align*}
|h(X_t)&-\P_{X_t}[h(X_{T_A})]-r_q(0)\P_{X_t}(T_A>\bm{e}_q)|\\
&=|h(X_t)-\P_{X_t}[h(X_{T_A})]-h_q(X_t)+\P_{X_t}[h_q(X_{T_A})e^{-qT_A}]|\\
&\le |h(X_t)-h_q(X_t)|+\P_{X_t}[|h(X_{T_A})-h_q(X_{T_A})|]\le 4M.
\stepcounter{equation}\tag{\theequation}
\end{align*}
Therefore, by the dominated convergence theorem, we obtain the assertion  (\ref{e2}).
\end{proof}

We give the proof of the main theorem for conditioning to avoid a bounded $F_\sigma$-set. 

\begin{proof}[The proof of Theorem \ref{t1.3}]
It holds by Theorems \ref{t2}, \ref{ep1}, and \ref{ep2}.
\end{proof}

We define 
\begin{align}
\P_{x;A}|_{\F_t}:=\frac{\varphi_A(X_t)}{\varphi_A(x)}1_{\{T_{A}>t\}}\cdot \P_x|_{\F_t}
\end{align}
for $x\in \{x\in \R;\ \varphi_A(x)>0\}$. In the same way as (\ref{c3}), we obtain
\begin{align}
\P_{x;A}(T_{A}=\infty)=1.
\end{align}
Similarly to Subsection \ref{SS1}, we obtain the following theorem:

\begin{thm}
It holds that
\begin{align}
\P_{x;A}=\P_{x;0}^{(0)}(\cdot|\ T_A=\infty).
\end{align}

\end{thm}

%%%%%%%%%%%%%%%%%%%%%%%%%%%%%%%%%%%%%%%%%%%%%%%%%%%%%%%%%%%%%%%%%%%%%%%%%%%%%%%%%%%%%%%%%%%%%%%%%%%%%%%%%%%%%%%%%%%%%%%%%%%%%%%%%%%%%%%%%%%%%%%%%%%%%%%%%%%%%%%%%%%%%%%%%%%%%%%%%%%%%%%%%%%%%%%%%%%%%%%%%%%%%%%%%%%%%%%%%%%%%%%%%%%%%%%%%%

\section{Conditioning to avoid an integer lattice}
\label{S6}
Let $L>0$ and set $L\Z:=\{Ln\in \R;\ n\in \Z\}$. In this Section, we further assume the following:
\begin{align}
\label{f1}
\sum_{n\in \Z}\left|\frac{1}{q+\Psi(\frac{2n\pi}{L})}\right|<\infty\qquad \text{for}\ q>0.
\end{align}
We define for $x\in \R$ and $q>0$,
\begin{align}
R_q(x):=\sum_{n\in \Z}\frac{e^{\frac{2n\pi x}{L}i}}{q+\Psi(\frac{2n\pi}{L})}
\end{align}
Note that $\Psi(u)=0$ and $u=0$ are equivalent. Isozaki \cite{Iso} showed 
\begin{align}
\label{f1-1}
\P_x[e^{-qT_{L\Z}}]=\frac{R_q(x)}{R_q(0)}.
\end{align}
We define for $x\in \R$ and $q>0$, $H_q(x):=R_q(0)-R_q(x)$. Then, the limit
\begin{align}
\label{f2}
\lim_{q\to 0+}H_q(x)=\sum_{n\in \Z\setminus \{0\}}\frac{1-e^{\frac{2n\pi x}{L}i}}{\Psi(\frac{2n\pi}{L})}
\end{align}
exists and is finite. Furthermore, if we let $\varphi_{L\Z}(x)$ denote this limit, it is continuous. Indeed, since $|\Psi(u)|\to\infty$ as $|u|\to \infty$, we have $|\frac{\Psi(u)}{q+\Psi(u)}|\to 1$ as $|u|\to \infty$. Thus, the convergence $\sum_{n\in \Z}|\frac{1}{q+\Psi(\frac{2n\pi}{L})}|$ and $\sum_{n\in \Z\setminus\{0\}}|\frac{1}{\Psi(\frac{2n\pi}{L})}|$ is equivalent, and then they converge by assumption (\ref{f1}). Therefore, we obtain by the dominated convergence theorem, 
\begin{align}
H_q(x)&=R_q(0)-R_q(x)=\sum_{n\in \Z\setminus \{0\}}\frac{1-e^{\frac{2n\pi x}{L}i}}{q+\Psi(\frac{2n\pi}{L})}\to \sum_{n\in \Z\setminus \{0\}}\frac{1-e^{\frac{2n\pi x}{L}i}}{\Psi(\frac{2n\pi}{L})}
\end{align}
as $q\to 0+.$

Note that the right hand side of (\ref{f2}) is the Fourier series of $\varphi_{L\Z}(x)$. Moreover, $\varphi_{L\Z}(x)$ is a periodic function with period $L$ and is also an even function.

First, we prove of the a.s. convergence.

\begin{prop}
\label{fp1}
For $x\in \R$, it holds that
\begin{align}
\lim_{q\to 0+}R_q(0)\P_x(T_{L\Z}>\bm{e}_q)=\varphi_{L\Z}(x).
\end{align}
\end{prop}

\begin{proof}
By Fubini's theorem and by (\ref{f1-1}) and (\ref{f2}), we have
\begin{align}
R_q(0)\P_x(T_{L\Z}>\bm{e}_q)=R_q(0)(1-\P_x[e^{-qT_{L\Z}}])=R_q(0)-R_q(x)\to \varphi_{L\Z}(x)
\end{align}
as $q\to 0+$. The proof is complete.
\end{proof}

\begin{Rem}
Since $\sum_{n\in \Z\setminus \{0\}}\frac{1}{\Psi(\frac{2n\pi}{L})}$ converges, we have
\begin{align}
\label{6.8}
qR_q(0)=1+\sum_{n\in \Z\setminus \{0\}}\frac{q}{q+\Psi(\frac{2n\pi}{L})}\to 1
\end{align}
as $q\to 0+.$ By (\ref{b8}), we obtain
\begin{align}
\lim_{q\to 0+}\frac{r_q(0)}{R_q(0)}=\lim_{q\to 0+}\frac{qr_q(0)}{qR_q(0)}= 0.
\end{align}
Therefore, comparing Proposition \ref{ep1} and Proposition \ref{fp1}, the speed of convergence of $\P_x(T_{A}>\bm{e}_q)$ and $\P_x(T_{L\Z}>\bm{e}_q)$ is different, where $A$ is a bounded $F_\sigma$-set.
\end{Rem}

Next, we prove of the $L^1(\P_x)$-convergence.

\begin{prop}
\label{fp2}
For $x\in \R$, it holds that
\begin{align}
\lim_{q\to 0+}R_q(0)\P_{X_t}(T_{L\Z}>\bm{e}_q)=\varphi_{L\Z}(X_t)\ \text{in}\ L^1(\P_x).
\end{align}
\end{prop}
\begin{proof}
We want to show $H_q(X_t)\to \varphi_{L\Z}(X_t)$ in $L^1(\P_x)$, but this is clear 
 by $H_q(X_t)\le \varphi_{L\Z}(X_t)\in L^1(\P_x)$ and by the dominated convergence theorem.
\end{proof}

We give the proof of the main theorem for conditioning to avoid an integer lattice. 

\begin{proof}[The proof of Theorem \ref{1.4}]
We define
\begin{align}
N_t^q:=R_q(0)\P_x(T_{L\Z}>\bm{e}_q>t|\F_t)
\end{align}
for $q>0$. By the lack of memory property of the exponential distribution and by the Markov property, we have
\begin{align*}
N_{t}^q&=R_q(0)1_{\{T_{L\Z}>t\}}\P_x^{\F_t}(T_{L\Z}>\bm{e}_q|\ \bm{e}_q>t)\P_x(\bm{e}_q>t)\\
&=R_q(0)1_{\{T_{L\Z}>t\}}\P_{X_t}(T_{L\Z}>\bm{e}_q)e^{-qt},
\stepcounter{equation}\tag{\theequation}
\end{align*}
where $\P_x^{\F_t}$ means the conditional probability with respect to $\F_t$. Thus, we have by Proposition \ref{fp1} and \ref{fp2},
\begin{align}
\label{ff1}
\lim_{q\to 0+}N_{t}^{q}=\varphi_{L\Z}(X_t)1_{\{T_{L\Z}>t\}}\ \mathrm{a.s.\ and\ in}\ L^1(\mathbb{P}_x).
\end{align}
Therefore, we obtain by Proposition \ref{fp1},
\begin{align*}
\P_x[F_t,\ t<\bm{e}_q|\ \bm{e}_q<T_{L\Z}]=\frac{\P_x[F_t\cdot N_t^q]}{R_q(0)\P_x(\bm{e}_q<T_{L\Z})}\to \mathbb{P}_x\left[F_t \cdot\frac{\varphi_{L\Z}(X_t)}{\varphi_{L\Z}(x)}1_{\{T_{L\Z}>t\}}\right],
\stepcounter{equation}\tag{\theequation}
\end{align*}
as $q\to 0+.$ The proof is complete.
\end{proof}

\begin{Rem}
We define
\begin{align}
M_t^q:=R_q(0)\P_x(\bm{e}_q<T_{L\Z}|\F_t)
\end{align}
for $q>0$. Since we have by (\ref{6.8}),
\begin{align*}
\label{7.16}
M_{t}^q-N_{t}^q&=R_q(0)\P_x(T_{L\Z}>\bm{e}_q,\ t\ge \bm{e}_q|\F_t)\\
&=qR_q(0)\int_0^t 1_{\{T_{L\Z}>s\}}e^{-qs}ds\\
&\to \int_0^t 1_{\{T_{L\Z}>s\}}ds>0\ \text{a.s. and in }L^1(\P_x),
\stepcounter{equation}\tag{\theequation}
\end{align*}
as $q\to 0+$. Thus, by (\ref{ff1}), we have
\begin{align*}
\lim_{q\to 0+}M_t^q&=\lim_{q\to 0+}(N_t^q+(M_t^q-N_t^q))\\
&=\varphi_{L\Z}(X_t)1_{\{T_{L\Z}>t\}}+\int_0^t 1_{\{T_{L\Z}>s\}}ds\ \mathrm{a.s.\ and\ in}\ L^1(\mathbb{P}_x).
\stepcounter{equation}\tag{\theequation}
\end{align*}
Therefore, we obtain
\begin{align*}
\P_x[F_t|\ \bm{e}_q<T_{L\Z}]&=\frac{\P_x[F_t\cdot M_t^q]}{R_q(0)\P_x(\bm{e}_q<T_{L\Z})}\\
&\to \P_x\left[F_t\cdot \frac{\varphi_{L\Z}(X_t)1_{\{T_{L\Z}>t\}}+\int_0^t 1_{\{T_{L\Z}>s\}}ds}{\varphi_{L\Z}(x)}\right]
\stepcounter{equation}\tag{\theequation}
\end{align*}
as $q\to 0+.$
\end{Rem}

We define 
\begin{align}
\P_{x;L\Z}(A,\ t<\zeta):=\P_x\left[1_A\cdot \frac{\varphi_{L\Z}(X_t)}{\varphi_{L\Z}(x)}1_{\{T_{L\Z}>t\}}\right]\qquad(A\in \F_t)
\end{align}
for $x\in \{x\in \R;\ \varphi_{L\Z}(x)>0\}$, where $\zeta$ denotes the life time. In the same way as (\ref{c3}), we obtain
\begin{align}
\P_{x;L\Z}(T_{L\Z}<\infty)=0.
\end{align}

%%%%%%%%%%%%%%%%%%%%%%%%%%%%%%%%%%%%%%%%%%%%%%%%%%%%%%%%%%%%%%%%%%%%%%%%%%%%%%%%%%%%%%%%%%%%%%%%%%%%%%%%%%%%%%%%%%%%%%%%%%%%%%%%%%%%%%%%%%%%%%%%%%%%%%%%%%%%%%%%%%%%%%%%%%%%%%%%%%%%%%%%%%%%%%%%%%%%%%%%%%%%%%%%%%%%%%%%%%%%%%%%%%%%%%%%%%%%%%%%%%%%%%%%%%%%%%%%%%%%%%%%%%%%%%%%%%%%%

\section*{Acknowledgement}
In writing this paper, the author would like to thank Professor Kouji Yano of the Graduate School of Science, Osaka University, for his careful guidance and great support. This work was supported by JST SPRING, Grant Number JPMJSP2138. 

\bibliographystyle{plain}

\begin{thebibliography}{10}
\bibitem{Ber}
J. Bertoin.
\newblock {\em L\'{e}vy processes}, volume 121 of {\em Cambridge Tracts in Mathematics}.
\newblock Cambridge University Press, Cambridge, 1996.

\bibitem{Chaumont}
L.~Chaumont.
\newblock Conditionings and path decompositions for {L}\'evy processes.
\newblock {\em Stochastic Process. Appl.}, 64(1):39--54, 1996.

\bibitem{CD}
L.~Chaumont and R.~A. Doney.
\newblock On {L}\'{e}vy processes conditioned to stay positive.
\newblock {\em Electron. J. Probab.}, 10:no. 28, 948--961, 2005.

\bibitem{CW}
K.~L. Chung and R.~J. Williams.
\newblock {\em Introduction to stochastic integration}.
\newblock Modern Birkh\"auser Classics. Birkh\"auser/Springer, New York, second edition, 2014.

\bibitem{DKW}
L. D\"oring, A.~E. Kyprianou, and P. Weissmann.
\newblock Stable processes conditioned to avoid an interval.
\newblock {\em Stochastic Process. Appl.}, 130(2):471--487, 2020.

\bibitem{DWW}
L. D\"oring, A.~R. Watson, and P. Weissmann.
\newblock L\'evy processes with finite variance conditioned to avoid an interval.
\newblock {\em Electron. J. Probab.}, 24:Paper No. 55, 32, 2019.

\bibitem{Durret}
R. Durrett.
\newblock {\em Probability---theory and examples}, volume~49 of {\em Cambridge Series in Statistical and Probabilistic Mathematics}.
\newblock Cambridge University Press, Cambridge, fifth edition, 2019.

\bibitem{GR}
T. Grzywny and M. Ryznar.
\newblock Hitting times of points and intervals for symmetric {L}\'evy processes.
\newblock {\em Potential Anal.}, 46(4):739--777, 2017.

\bibitem{IY}
K. Iba and K. Yano.
\newblock Two-point local time penalizations with various clocks for {L}\'{e}vy processes.
\newblock preprint, arXiv:2404.06759.

\bibitem{Iso}
Y. Isozaki.
\newblock The first hitting time of the integers by symmetric {L}\'evy processes.
\newblock {\em Stochastic Process. Appl.}, 129(5):1782--1794, 2019.

\bibitem{Kn}
F.~B. Knight.
\newblock Brownian local times and taboo processes.
\newblock {\em Trans. Amer. Math. Soc.}, 143:173--185, 1969.

\bibitem{KRS}
A.~E. Kyprianou, V. Rivero, and B.~Seng\"ul.
\newblock Conditioning subordinators embedded in {M}arkov processes.
\newblock {\em Stochastic Process. Appl.}, 127(4):1234--1254, 2017.

\bibitem{Lam}
A.~Lambert.
\newblock Completely asymmetric {L}\'evy processes confined in a finite interval.
\newblock {\em Ann. Inst. H. Poincar\'e{} Probab. Statist.}, 36(2):251--274, 2000.

\bibitem{LW}
P. Lenthe and P. Weissmann.
\newblock Completely asymmetric stable processes conditioned to avoid an interval.
\newblock {\em J. Appl. Probab.}, 56(4):1187--1197, 2019.

\bibitem{Pa}
H. Pant\'{\i}.
\newblock On {L}\'{e}vy processes conditioned to avoid zero.
\newblock {\em ALEA Lat. Am. J. Probab. Math. Stat.}, 14(2):657--690, 2017.

\bibitem{Sato}
K. Sato.
\newblock {\em L\'{e}vy processes and infinitely divisible distributions}, volume~68 of {\em Cambridge Studies in Advanced Mathematics}.
\newblock Cambridge University Press, Cambridge, 1999.
\newblock Translated from the 1990 Japanese original, Revised by the author.

\bibitem{Takeda}
S. Takeda.
\newblock Sample path behaviors of l\'{e}vy processes conditioned to avoid zero, preprint, arXiv:2211.12863.

\bibitem{TY}
S. Takeda and K. Yano.
\newblock Local time penalizations with various clocks for {L}\'{e}vy processes.
\newblock {\em Electron. J. Probab.}, 28:Paper No. 12, 35, 2023.

\bibitem{Tukada}
H. Tsukada.
\newblock A potential theoretic approach to {T}anaka formula for asymmetric {L}\'{e}vy processes.
\newblock In {\em S\'{e}minaire de {P}robabilit\'{e}s {XLIX}}, volume 2215 of {\em Lecture Notes in Math.}, pages 521--542. Springer, Cham, 2018.

\bibitem{Yano2}
K. Yano.
\newblock On harmonic function for the killed process upon hitting zero of asymmetric {L}\'{e}vy processes.
\newblock {\em J. Math-for-Ind.}, 5A:17--24, 2013.

\bibitem{Yano}
K. Yano.
\newblock On universality in penalisation problems with multiplicative weights.
\newblock In {\em Dirichlet forms and related topics}, volume 394 of {\em Springer Proc. Math. Stat.}, pages 535--558. Springer, Singapore, [2022] \copyright 2022.

\end{thebibliography}

\end{document}